\pdfoutput=1
\documentclass[11pt,a4paper]{article}

\usepackage[utf8]{inputenc}

\usepackage{enumitem}

\usepackage{setspace}

\usepackage{amsmath}
\usepackage{amssymb}
\usepackage{amsfonts}
\usepackage{amsthm}

\usepackage{mathrsfs}

\usepackage{array}

\usepackage{xcolor}

\newcommand{\mybibliography}{\bibliography{refs-bibtex}}

\usepackage[page,toc]{appendix}

\usepackage[T1]{fontenc}
\usepackage{lmodern}

\usepackage{upquote}

\usepackage{listings}

\usepackage{longtable}

\usepackage{afterpage}

\usepackage{algpseudocode}
\makeatletter
\algnewcommand{\ParState}[1]{\State\parbox[t]{\dimexpr\linewidth-\the\ALG@thistlm}{\strut #1\strut}}
\makeatother
\algnewcommand{\LineComment}[1]{\vspace{1em}\ParState{\textit{(#1)}}}

\usepackage{tikz,pgfplots}
\usetikzlibrary{fit,arrows,arrows.meta,positioning}
\pgfplotsset{compat=1.3}

\usepackage{rotating}

\usepackage{adjustbox}

\usepackage{hyperref}

\usepackage{floatrow}
\newcommand{\code}[1]{%
\begingroup%
\let\oldunderscore\_%
\renewcommand{\_}{\discretionary{\oldunderscore}{}{\oldunderscore}}%
\renewcommand{\-}{\discretionary{\textrm{-}}{}{}}%
\texttt{#1}%
\endgroup}

\newcommand{\define}[1]{\textbf{#1}}

\usepackage[capitalize,noabbrev]{cleveref}

\newcommand{\CC}{\mathbb{C}}
\newcommand{\FF}{\mathbb{F}}
\newcommand{\NN}{\mathbb{N}}

\newcommand{\QQ}{\mathbb{Q}}

\newcommand{\ZZ}{\mathbb{Z}}

\newcommand{\cF}{\mathcal{F}}
\newcommand{\cG}{\mathcal{G}}

\newcommand{\cK}{\mathcal{K}}
\newcommand{\cL}{\mathcal{L}}

\newcommand{\cO}{\mathcal{O}}
\newcommand{\cP}{\mathcal{P}}

\newcommand{\cS}{\mathcal{S}}

\newcommand{\cW}{\mathcal{W}}
\newcommand{\cX}{\mathcal{X}}
\newcommand{\cY}{\mathcal{Y}}
\newcommand{\cZ}{\mathcal{Z}}

\newcommand{\fp}{\mathfrak{p}}

\newcommand{\sU}{\mathscr{U}}

\newcommand{\Gal}{\operatorname{Gal}}
\newcommand{\Factors}{\operatorname{Factors}}
\newcommand{\Orbits}{\operatorname{Orbits}}
\newcommand{\LCM}{\operatorname{lcm}}

\newcommand{\Stab}{\operatorname{Stab}}
\newcommand{\Aut}{\operatorname{Aut}}
\newcommand{\Fix}{\operatorname{Fix}}

\newcommand{\val}{\operatorname{val}}

\newcommand{\orb}{{\operatorname{orb}}}

\newcommand{\isom}{\cong}

\newcommand{\subgrp}{\leq}
\newcommand{\subgrpne}{<}
\newcommand{\sdp}{\rtimes}
\newcommand{\suchthat}{\,:\,}

\newcommand{\parens}[1]{\left(#1\right)}
\newcommand{\braces}[1]{\left\lbrace#1\right\rbrace}
\newcommand{\angles}[1]{\left\langle#1\right\rangle}

\newcommand{\abs}[1]{\left\lvert#1\right\rvert}

\def\comment{}
\def\endcomment{}
\long\def\comment#1\endcomment{}

\theoremstyle{plain}
\newtheorem{lemma}{Lemma}[section]

\theoremstyle{definition}
\newtheorem{algorithm}[lemma]{Algorithm}
\newtheorem{definition}[lemma]{Definition}
\newtheorem{example*}[lemma]{Example} 
  {%
   \pushQED{\qed}\begin{example*}}
  {\popQED\end{example*}}
\theoremstyle{remark}
\newtheorem{remark}[lemma]{Remark}

\makeatletter
\renewcommand*{\verbatim@font}{\ttfamily\fontseries{m}\selectfont}
\makeatother

\lstdefinelanguage{Magma}{
  morekeywords={end,function,intrinsic,procedure,for,while,repeat,until,do,in,if,else,elif,then,error,assert,require,when,where,is,print,printf,vprint,vprintf,time,declare,verbose,type,attributes,return,continue,break,delete,loop},
  morekeywords=[2]{eq,ne,le,lt,ge,gt,cmpeq,cmpne,not,notin,and,or,notsubset,subset,meet,join,diff,sdiff,assigned,eval},
  morekeywords=[3]{sub,ncl,func,proc,ideal,elt},
  morekeywords=[4]{AnyPadExact,StrPadExact,PadExactElt,FldPadExact,FldPadExactElt,RngUPol_FldPadExact,RngUPolElt_FldPadExact,RngMPol_FldPadExact,RngMPolElt_FldPadExact,SetCart_PadExactElt,Tup_PadExactElt,Val_PadExactElt,Val_FldPadElt,Val_RngUPolElt_FldPad,Val_RngMPolElt_FldPad,RngInt,RngIntElt,SetCart,Tup,List,FldNum,FldNumElt,FldRat,FldRatElt,FldPad,FldPadElt,Getter,BoolElt,SetCart_PadExact,Tup_PadExact},
  sensitive=true,
  morecomment=[l]{//},
  morecomment=[s]{/*}{*/},
  morestring=[b]",
}
\theoremstyle{definition}
\newtheorem*{parameterlist*}{Parameters}

\begin{document}

\title{Computing the Galois group of a polynomial\\over a \(p\)-adic field}
\author{\href{https://cjdoris.github.io}{Christopher Doris} \\ Heilbronn Institute of Mathematical Research \\ \href{mailto:Christopher Doris <christopher.doris@bristol.ac.uk>?subject=Computing the Galois group of a polynomial over a p-adic field}{\texttt{christopher.doris@bristol.ac.uk}}}
\date{March 2020}
\maketitle

\begin{abstract}
We present a family of algorithms for computing the Galois group of a polynomial defined over a \(p\)-adic field. Apart from the ``naive'' algorithm, these are the first general algorithms for this task. As an application, we compute the Galois groups of all totally ramified extensions of \(\mathbb Q_2\) of degrees 18, 20 and 22, tables of which are available online.
\end{abstract}

\section{Introduction}

In this article we consider the following problem, the \(p\)-adic instance of the forward Galois problem: given a \(p\)-adic field \(K\) and a polynomial \(F(x) \in K[x]\) over that field, what is its Galois group \(G := \Gal(F/K)\)?

Over any field for which polynomial factorization algorithms are known, the forward Galois problem can always be solved with the \define{naive algorithm}: explicitly compute the splitting field of \(F\) by repeatedly adjoining a root of it to the base field, and then explicitly compute the automorphisms of the splitting field. To date, there is no general solution to the \(p\)-adic forward Galois problem other than the naive algorithm.

This article presents a general algorithm. In practice, it can for example quickly determine the Galois group of most irreducible polynomials of degree 16 over \(\QQ_2\) and has been used to compute some non-trivial Galois groups at degree 32. It has been tested on polynomials defining all extensions of \(\QQ_2\), \(\QQ_3\) and \(\QQ_5\) of degree up to 12, all extensions of \(\QQ_2\) of degree 14, and all totally ramified extensions of \(\QQ_2\) of degrees 18, 20 and 22, the latter three being new. See \cref{gg-sec-implementation}.

Our implementation is publicly available \cite{galoiscode} and pre-computed tables of Galois groups are available from here also.

\subsection{Overview of algorithm}

Our algorithm uses the ``resolvent method''. We now describe a concrete instance.

Suppose \(F(x) \in \QQ_p[x]\) is irreducible of degree \(d\), and therefore defines an extension \(L/\QQ_p\) of degree \(d\).

The ramification filtration of this extension is a tower \(L_t=L/\ldots/L_0=\QQ_p\). Let \(F_1(x)\in \QQ_p[x]\) be a defining polynomial for \(L_1/\QQ_p\). By Krasner's lemma, any polynomial in \(\QQ[x]\) sufficiently close to \(F_1\) is also a defining polynomial, so we may take \(F_1 \in \QQ[x]\). It is irreducible and so defines the number field \(\cL_1/\cL_0=\QQ\) which has a unique completion embedding into \(L_1\). Repeating this procedure up the tower, we obtain the tower of number fields \(\cL=\cL_t/\ldots/\cL_0=\QQ\) such that \(\cL\) embeds uniquely into \(L\). We call \(\cL/\QQ\) a \define{global model} of \(L/\QQ_p\).

Let \(d_i:=(\cL_i:\cL_{i-1})=(L_i:L_{i-1})\), then \(\Gal(\cL_i/\cL_{i-1}) \subgrp S_{d_i}\) and therefore \(\Gal(\cL/\QQ) \subgrp W := S_{d_t} \wr \cdots \wr S_{d_1}\). Observe also that naturally \(\Gal(L/\QQ_p) \subgrp \Gal(\cL/\QQ)\) since the left hand side is a decomposition group of the right hand side.

Suppose \(\alpha_1 \in \cL\) generates \(\cL/\QQ\), and let \(\alpha_2,\ldots,\alpha_d\in\bar\QQ\) be its \(\QQ\)-conjugates. Suppose we choose some subgroup \(U \subgrp W\), find an \define{invariant} \(I \in \ZZ[x_1,\ldots,x_d]\) such that \(\Stab_W(I) = U\) and compute the \define{resolvent}
\[R(x) = \prod_{wU \in W/U}(t - wU(I)(\alpha_1,\ldots,\alpha_d)) \in \ZZ[t]\]
by finding sufficiently precise complex approximations to \(\alpha_1,\ldots,\alpha_d\), giving a complex approximation to \(R\), whose coefficients we can then round to \(\ZZ\).

One can show that \(\Gal(R/\QQ) = q(\Gal(\cL/\QQ))\) and hence \(\Gal(R/\QQ_p) = q(\Gal(L/\QQ_p)) = q(\Gal(F/\QQ_p))\) where \(q:W\to S_{W/U}\) is the action of \(W\) on the cosets of \(U\).

In particular, if we define \(s(G)\) to be the multiset of the sizes of orbits of the permutation group \(G\), and we let \(S\) be the multiset of the degrees of the factors of \(R\) over \(K\), then \(s(q(\Gal(F/\QQ_p))) = S\).

We compute the set \(\cG\) of all transitive subgroups of \(W\), so that \(\Gal(F/\QQ_p) \in \cG\). If \(\abs{\cG}>1\), we search through the subgroups \(U \subgrp W\) in index order until we find one such that \(\{s(q(G)) \suchthat G \in \cG\}\) contains at least two elements. We then compute the corresponding resolvent \(R(t) \in \ZZ[t]\), factorize it over \(\QQ_p\) and let \(S\) be the multiset of degrees of factors, and replace \(\cG\) by \(\{G\in\cG \suchthat s(q(G)) = S\}\). Observe that \(\cG\) is now strictly smaller than it was before, and we still have \(\Gal(F/\QQ_p) \in \cG\).

We repeat this process until \(\abs{\cG}=1\), at which point this single group is the Galois group and we are done.

In \cref{gg-sec-arm} we describe our precise formulation of this algorithm.

We have described one method of producing a global model, which results in the group \(W\) (relative to which we compute resolvents) being a wreath product of symmetric groups. It is better for \(W\) to be as small as possible, since this will reduce the index \((W:U)\) required, and hence also reduce \(\deg R\). In \cref{gg-sec-glomod} we discuss some other constructions. The best constructions take advantage of the simple structure of the Galois group of a ``singly ramified'' extension, something like \(C_d\) for unramified extensions, \(C_d \sdp (\ZZ/d\ZZ)^\times\) for tame extensions and \(C_p^k \sdp H\) for wild extensions. We can also produce global models for reducible \(F\) using global models for its factors.

In this example, we deduced the Galois group by enumerating the set \(\cG\) of all possibilities and then eliminating candidates. This is the ``group theory'' part of the algorithm. We have other methods which avoid enumerating all subgroups of \(W\), and instead work down the graph of subgroups of \(W\). These are discussed in \cref{gg-sec-groups}.

The function \(s\) taking a group and returning the multiset of sizes of its orbits is a ``statistic'', and there are other choices. These are discussed in \cref{gg-sec-statistic}. Some statistics provide more information than others, and therefore can result in smaller indices \((W:U)\) being required, but this comes at the expense of taking longer to compute.

We search for \(U\) by enumerating all the subgroups of \(W\) of each index in turn until we find one which is useful. There are other methods which try to avoid computing all of these subgroups, of which there may be many. One method restricts to a special class of subgroups. These are given in \cref{gg-sec-choice}.

\subsection{Previous work}

Over \(p\)-adic fields, there are some special cases where Galois groups can be computed.
\begin{itemize}
\item It is well known that the unramified extensions of \(K\) of degree \(d\) are all isomorphic, Galois and have cyclic Galois group \(C_d\). Hence if the irreducible factors of \(F(x)\) all define unramified extensions, then the splitting field of \(F(x)\) is unramified, Galois and cyclic with degree \(\LCM \{\deg g \,:\, g \in \Factors(F)\}\).
\item Suppose \(L/K\) is tamely ramified. Then it has a maximal unramified subfield \(U\), and \(L/U\) is totally (tamely) ramified. It is well known that \(L = U(\sqrt[e]{\zeta^r \pi})\) where \(e = (L:U)\) for some uniformizer \(\pi \in K\), \(\zeta\) a root of unity generating \(U\) and \(r \in \ZZ\). In this special form, it is straightforward to write down the splitting field and Galois group of \(L/K\). Furthermore, it is easy to compute the compositum of tame extensions, and hence if each irreducible factor of \(F(x)\) defines a tamely ramified extension, we can compute its Galois group. See \cite[Ch. II, \S2.2]{DPhD} for an exposition.
\item Greve and Pauli have studied \define{singly ramified} extensions, that is extensions whose ramification polygon has a single face, giving an explicit description of their splitting field and Galois group \cite[Alg. 6.1]{GP}. So in particular if \(F(x)\) is an Eisenstein polynomial whose ramification polygon has a single face, then we can compute its Galois group. An explicit description of this algorithm appears in Milstead's thesis \cite[Alg. 3.23]{Mil}.
\item In his thesis, Greve extends this to an algorithm for \define{doubly ramified} extensions \cite[\S6.3]{GreveTh}, that is whose ramification polygon has two faces. Essentially this uses the singly ramified algorithm for the bottom part, and class field theory and group cohomology to deal with the elementary abelian top part.
\item Jones and Roberts \cite{LFDB} have computed all extensions of \(\QQ_p\) of degree up to 12, including their Galois group and some other invariants. These are available online in the Local Fields Database (LFDB). Some of the methods they use to compute Galois groups will feature in our general algorithm.
\item Awtrey et al. have also considered degree 12 extensions of \(\QQ_2\) and \(\QQ_3\) \cite{AwtreyTh}; degree 14 extensions of \(\QQ_2\) \cite{AwtreyD14}; degree 15 extensions of \(\QQ_5\) \cite{AwtreyD15}; and degree 16 \emph{Galois} extensions of \(\QQ_2\) \cite{AwtreyD16}. The main new idea in these articles is the \define{subfield Galois group content} of an extension \(L/K\): the set of Galois groups of all proper subfields of \(L/K\). This invariant of \(\Gal(L/K)\) is useful in distinguishing between possible Galois groups, and is possible to compute given a database of all smaller extensions.
\end{itemize}

The difficult case appears to be when the factors of \(F\) define wildly ramified extensions whose ramification polygons have many faces.

Recently Rudzinski has developed techniques for evaluating linear resolvents \cite{Rudz} and Milstead has used a combination of these techniques with the ones mentioned above to compute some Galois groups in this difficult class \cite{Mil}.

\subsection{Mathematical notation}

Roman capital letters \(K,L,\ldots\) denote \(p\)-adic fields. The ring of integers of \(K\) is denoted \(\cO_K\), a uniformizer is denoted \(\pi_K\) and the residue class field is denoted \(\FF_K = \cO_K/(\pi_K)\). If \(u \in \cO_K\) then \(\bar u = u+(\pi_K) \in \FF_K\) is its residue class. We denote by \(v_K\) the valuation of \(\bar\QQ_p\) such that \(v_K(\pi_K)=1\).

Calligraphic capital letters \(\cK,\cL,\ldots\) denote number fields. The ring of integers of \(\cK\) is \(\cO_\cK\).

If \(U \subgrp W\) is a subgroup then \(q_U:W \to S_{W/U}\) denotes the action of \(W\) on the left cosets of \(U\).

As introduced in \cref{gg-sec-statistic}, \(s\) denotes a function whose input is a permutation group or a polynomial and whose output is anything. There is an equivalence relation \(\sim\) on outputs such that if \(F(x) \in K[x]\) then \(s(\Gal(F))\sim s(F)\). There may also be a partial ordering \(\preceq\) on outputs such that if \(H \subgrp G\) are groups then \(s(H) \preceq s(G)\).

We may omit subscripts from the notation if they are clear from context.

\subsection{A note on conjugacy}

Recall that the Galois group of a polynomial \(G = \Gal(F)\) is defined to be the group of automorphisms of the splitting field of \(F\). Usually, we represent this as a permutation group \(G \subgrp S_d\) where \(d = \deg(F)\), such that writing the roots of \(F\) as \(\alpha_1,\ldots,\alpha_d\) in some order, then \(G\) acts as \(g(\alpha_i) = \alpha_{g(i)}\).

Since the order of the roots was arbitrary, \(G\) is only really defined up to conjugacy in \(S_d\).

Sometimes, we may know more about the roots of \(F\). For instance, if \(F\) is reducible, then \(G\) has multiple orbits. If we explicitly factorize \(F = \prod_i F_i\), and let \(d_i = \deg(F_i)\), then we can specify that the first \(d_1\) roots \(\alpha_1,\ldots,\alpha_{d_1}\) are the roots of \(F_1\), the next \(d_2\) are the roots of \(F_2\) and so on. Letting \(W = S_{d_1} \times S_{d_2} \times \ldots\) then \(G \subgrp W \subgrp S_d\) is defined up to conjugacy in \(W\). We shall see more examples in \cref{gg-sec-glomod}.

Almost everywhere in our exposition, when we talk of a group, we actually mean the conjugacy class of the group inside some understood larger group. When we talk of the collection of all groups with some property, we mean all the conjugacy classes whose groups have that property. This is to simplify the exposition.

In the implementation, a conjugacy class is usually represented by a representative group. An algorithm which returns all conjugacy classes with some property may actually return several representatives for the same class. Finding which groups generate the same class in order to remove duplicates can be computationally difficult, and so whether or not to do this, and how, is usually parameterised. The default is not to remove duplicates. See \cite[Ch. II, \S 11]{DPhD} for details.

Henceforth, we shall typically only mention conjugacy when we have specific strategies to deal with conjugate groups.

\subsection{Compendium}
\label{gg-sec-tldr}

Most of the rest of this article describes in full detail the possible parameters to our algorithm, of which there are many. We now list the sections with the most important or novel contributions.

\begin{itemize}
\item \Cref{gg-sec-arm}: Describes the resolvent method, the main focus of this article.
\item \Cref{gg-sec-reseval,gg-sec-glomod}: Methods for producing ``global models'' for \(p\)-adic fields, which are used to evaluate resolvents. Our constructions are more general than previous similar efforts and so can produce more efficient models.
\item \Cref{gg-sec-groups-all,gg-sec-groups-maximal2}: The main two ways we perform the group theory part of deducing the Galois group. The former is to write down all possibilities and then eliminate until one remains; the latter works down the graph of possible groups using the notion of ``maximal preimages of statistics'' to efficiently move down the graph without blowing up the number of possibilities.
\item \Cref{gg-sec-stat-facdegs}: The main ``statistic'' of a resolvent we compute is the multiset of degrees of its factors. This is compared to the multiset of sizes of orbits of potential Galois groups to deduce which are possible.
\item \Cref{gg-sec-tranche-oidx}: Methods to produce groups from which to compute resolvents which empirically are both fast to compute and give low-degree resolvents.
\item \Cref{gg-sec-implementation}: The implementation, timings, performance notes, etc.
\end{itemize}

\section{Galois group algorithms}
\label{gg-sec-algorithms}

This article is mainly concerned with the resolvent method, introduced in \cref{gg-sec-arm}. However, the algorithm is recursive, in that it may compute other Galois groups along the way, and it may suffice to use other algorithms for this purpose. Therefore, we briefly describe the other algorithms available in our implementation.

\subsection{\texttt{Naive}}
\label{gg-sec-naive}

This explicitly computes a splitting field for \(F(x)\) and explicitly computes its automorphisms.

This is the algorithm currently implemented in Magma for \(p\)-adic polynomials, called \texttt{GaloisGroup}. Since the splitting field is computed explicitly, this is only suitable when the Galois group is known in advance to be small, such as because the degree is small.

\subsection{\texttt{Tame}}
\label{gg-sec-tame}

As explained in the introduction, if the irreducible factors of \(F(x)\) all generate tamely ramified extensions of \(K\), then its Galois group can be computed directly.

\subsection{\texttt{SinglyRamified}}
\label{gg-sec-singlyramified}

This computes the Galois group of \(F(x)\) provided it is irreducible and defines an extension whose ramification filtration contains a single segment. Such an extension is called \define{singly ramified}.

When the extension is tamely ramified, we can use the \texttt{Tame} algorithm. Otherwise the extension is totally wildly ramified and we use an algorithm due to Greve and Pauli \cite[Alg. 6.1]{GP}. An explicit description is given by Milstead \cite[Alg. 3.23]{Mil}.

\subsection{\texttt{ResolventMethod}}
\label{gg-sec-arm}

The resolvent method is the focus of the remainder of this article and is based on the following simple lemma.

\begin{lemma}
Suppose \(G := \Gal(F) \subgrp W \subgrp S_d\) where \(d = \deg F\), and take any \(U \subgrp W\). Now \(S_d\) acts on \(\ZZ[x_1,\ldots,x_d]\) by permuting the variables, so suppose \(I \in \ZZ[x_1,\ldots,x_n]\) such that \(\Stab_W(I) = U\) (we say \(I\) is a \define{primitive \(W\)-relative \(U\)-invariant}). Letting \(\alpha_1,\ldots,\alpha_d\) be the roots of \(F\), define \(\beta_{wU} = wU(I)(\alpha_1,\ldots,\alpha_n)\) (this is well-defined since \(I\) is fixed by \(U\)) and define the \define{resolvent} \(R(t) := \prod_{wU \in W/U} (t - \beta_{wU})\). Then \(R(t) \in K[t]\). If \(R\) is squarefree, then its Galois group corresponds to the coset action of \(G\) on \(U\).  That is, letting \(q : W \to S_{W/U}\) be the coset action, then identifying \(wU \leftrightarrow \beta_{wU}\) we have \(\Gal(R) = q(G)\).
\end{lemma}

\begin{proof}
Writing \(R(t) := \tilde R(\alpha_1,\ldots,\alpha_d; t)\) where \[\tilde R(x_1,\ldots,x_d; t) := \prod_{wU \in W/U} (t - wU(I)(x_1,\ldots,x_d))\] then the \(t\)-coefficients of \(\tilde R\) are fixed by \(W\) (the action of \(W\) re-orders the product) and hence by \(G\). We conclude that the \(t\)-coefficients of \(R\) are fixed by \(G\) too, and hence by Galois theory \(R(t) \in K[t]\).

If \(R\) is squarefree, then there is a 1-1 correspondence between the cosets \(\braces{wU}\) of \(W/U\) and the roots \(\braces{\beta_{wU}}\) of \(R\). Take \(g \in G\), then
\begin{align*}
g(\beta_{wU}) &= g(wU(I)(\alpha_1,\ldots,\alpha_d)) \\
&= wU(I)(g(\alpha_1),\ldots,g(\alpha_d))) \\
&= wU(I)(\alpha_{g(1)},\ldots,\alpha_{g(d)}) \\
&= gwU(I)(\alpha_1,\ldots,\alpha_d) \\
&= \beta_{gwU} \\
\end{align*}
so the action of \(G\) on the roots of \(R\) corresponds to the coset action, as claimed.
\end{proof}

Therefore, if we have some \(W\) containing \(G\) and a means to compute resolvents \(R\) for \(U \subgrp W\), then since \(\Gal(R) = q(G)\) is a function of \(G\), we can deduce information about \(G\) by finding some information about \(\Gal(R)\). Specifically how we compute resolvents and deduce information about \(G\) is controlled by two parameters.

Firstly, a resolvent evaluation algorithm (\cref{gg-sec-reseval}) selects a fixed group \(W \leq S_d\) such that \(G \leq W\), and thereafter is responsible for evaluating the resolvents \(R(t)\) from selected \(U \leq W\) and invariants \(I \in \ZZ[x_1,\ldots,x_d]\).

Secondly, a group theory algorithm (\cref{gg-sec-groups}) is responsible for deducing the Galois group \(G\) by choosing a suitable \(U\), and then using the resolvent \(R\) returned by the resolvent evaluation algorithm to gather information about \(G\).

\begin{algorithm}[Galois group: resolvent method] Given a polynomial \(F(x) \in K[x]\), returns its Galois group.\hfill
\label{gg-alg-arm}
\begin{algorithmic}[1]
\State Initialize the resolvent evaluation algorithm.\label{gg-alg-arm-rinit}
\State Initialize the group theory algorithm.\label{gg-alg-arm-ginit}
\State If we have determined the Galois group, then return it.\label{gg-alg-arm-done}
\State Let \(U\) be a subgroup of \(W\).\label{gg-alg-arm-U}
\State Let \(I\) be a primitive \(W\)-relative \(U\)-invariant.\label{gg-alg-arm-I}
\State Let \(R\) be the resolvent corresponding to \(I\).\label{gg-alg-arm-R}
\State Use \(R\) to deduce information about the Galois group.\label{gg-alg-arm-deduce}
\State Go to step \ref{gg-alg-arm-done}.
\end{algorithmic}
\end{algorithm}

The resolvent algorithm controls steps \ref{gg-alg-arm-rinit} and \ref{gg-alg-arm-R}. The group theory algorithm controls steps \ref{gg-alg-arm-ginit}, \ref{gg-alg-arm-done}, \ref{gg-alg-arm-U} and \ref{gg-alg-arm-deduce}. Step \ref{gg-alg-arm-I} could also be parameterised, but we find it is sufficient to use the algorithm due to Fieker and Kl\"uners \cite[\S5]{FK}, implemented as the intrinsic \texttt{RelativeInvariant} in Magma.

\begin{remark}
Using resolvents to compute Galois groups is not new. Stauduhar's method \cite{Stauduhar73} for polynomials over \(\QQ\) computes resolvents relative to \(S_d\) by computing complex approximations to the roots. This was improved by Fieker and Kl\"uners \cite{FK} to a ``relative resolvent method'' which allows the overgroup \(W\) to be made smaller at each iteration until it equals \(G\). Over \(\QQ_p\), a resolvent method has been used by Jones and Roberts \cite{LFDB} to compute the Galois group of fields of degree up to 12, computing resolvents in \(W = S_{d_2} \wr S_{d_1}\) corresponding to a subfield of degree \(d_1\).
\end{remark}

\subsection{\texttt{Sequence}}

This algorithm takes as parameters a sequence of other algorithms to compute Galois groups. It tries each algorithm in turn until one succeeds. This is mainly useful to deal with special cases first (e.g. \texttt{Tame} or \texttt{SinglyRamified}) before applying a general method (e.g. \texttt{ResolventMethod}).

\section{Resolvent evaluation algorithms}
\label{gg-sec-reseval}

These are used as part of the \texttt{ResolventMethod} algorithm for computing Galois groups. They are responsible for selecting an overgroup \(W\) such that \(G \subgrp W\) and thereafter evaluating resolvents relative to \(W\).

Currently there is one option, \code{Global}, described here.

\begin{definition}
A \define{global model} for a \(p\)-adic field \(K\) is an embedding \(i : \cK \to K\) where \(\cK\) is a global number field such that \(K\) is a completion of \(\cK\) and \(i\) is the corresponding embedding.

If \(L/K\) is an extension of \(p\)-adic fields, and \(i : \cK \to K\) is a global model for \(K\), then a \define{global model for \(L/K\) extending \(i\)} is a global model \(j : \cL \to L\) of \(L\) such that \(j|_{\cK} = i\).

Similarly a \define{global model for \(F(x) \in K[x]\) extending \(i\)} is \(\prod_k \cF_k\) where \(F = \prod_k F_k\) is the factorization over \(K\) of \(F\) into irreducible factors, \(L_k/K\) are the corresponding extensions, \(i_k : \cL_k \to L_k\) are global models for \(L_k/K\) extending \(i\), and \(\cL_k \isom \cK(x)/(\cF_k(x))\).

We shall often refer to \(\cK\) itself as the global model, instead of the embedding \(i\).
\end{definition}

The \texttt{Global} algorithm computes a global model \(\cK\) for \(K\) and a global model \(\cF(x) \in \cK[x]\) for the input \(F(x) \in K[x]\) extending \(\cK\). At the same time, it computes the required overgroup \(W\) such that \(G \subgrp \Gal(\cF / \cK) \subgrp W\). A parameter (a global model algorithm, \cref{gg-sec-glomod}) specifies how to produce a global model for \(F(x)\).

\begin{remark}
\label{gg-rmk-glomodidx}
Note that this implies that \(\deg\cF=\deg F=d\). In fact, our algorithm more generally computes an \define{overgroup embedding} \(e:W\to\cW\) such that \(G\subgrp W\), \(\Gal(\cF/\cK)\subgrp\cW\) and \(e(G)\) is the corresponding decomposition group. Hence \(\deg\cF>d\) is allowed. This usually arises as a global model \(\cL/\cK'/\cK\) for \(L/K\) where \(\cK'\) is also a global model for \(K\) and \((\cL:\cK')=d\), in which case we refer to \((\cK':\cK)\) as the \define{index} of the global model. In our exposition we shall assume \(W=\cW\) for simplicity and leave the details to \cite[Ch. II]{DPhD}.
\end{remark}

The algorithm then can evaluate resolvents as follows. For each complex embedding \(c : \cK \to \CC\), we compute the roots of \(c(\cF)\) to high precision. Letting \(\tilde\alpha_1,\ldots,\tilde\alpha_{d'}\) be these roots, we compute \[\tilde R_c(t) := \prod_{wU \in W/U}(t - wU(I)(\tilde\alpha_1,\ldots,\tilde\alpha_{d'}))\] which is an approximation to \(c(R(t)) \in \CC[t]\).

We can always arrange for \(\cF(x)\) to be monic and integral, so that its roots are integral, and therefore \(R(t) \in \cO_{\cK}[t]\). Firstly, suppose that \(\cK = \QQ\) (so \(K = \QQ_p\)), then we know \(R(t) \in \ZZ[t]\) and therefore assuming we have computed \(\tilde R(t)\) sufficiently precisely, then we can compute \(R(t)\) by rounding its coefficients to the nearest integer.

More generally, for each coefficient \(R_i\) of \(R(t)\) we take the vector \((\tilde R_{c,i})_c\) which should be a close approximation to \((c(R_i))_c\). Since \(R_i\) are integral, \((c(R_i))_c\) is an element of the \define{Minkowski lattice} \(\prod_c c(\cO_\cK)\), which is discrete, and therefore we can deduce \(R_i\) by rounding \((\tilde R_{c,i})_c\) to the nearest point in the lattice. This can be done using lattice basis reduction techniques such as LLL.

\begin{algorithm}[Resolvent: \texttt{Global}] Given a global model \(\cF(x) \in \cK[x]\) and subgroup \(U \subgrp W\), returns the corresponding resolvent \(R(t)\).
\label{gg-alg-resolvent}
\begin{algorithmic}[1]
\State Choose a Tschirnhaus transformation \(T \in \ZZ[x]\) (see Rmk. \ref{gg-rmk-resolvent-tschirnhaus}).
\State Choose a complex floating point precision, \(k\) decimal digits (see Rmk. \ref{gg-rmk-resolvent-complex-precision}).
\State Compute complex approximations to the roots of \(c(\cF)\) for each complex embedding \(c : \cK \to \CC\).
\State Compute \(\tilde R_c(t) = \prod_{wU \in W/U} (t - wU(I)(T(\tilde \alpha_1),\ldots,T(\tilde \alpha_{d'})))\).
\State Round \((\tilde R_{c,i})_i\) to the nearest point of the Minkowski lattice of \(\cO_{\cK}\), and let \(R_i\) be the corresponding element of \(\cO_{\cK}\).
\State If \(R(t) \in \cK[t]\) is not squarefree, go to Step 1.
\State Return \(R(t)\).
\end{algorithmic}
\end{algorithm}

\begin{remark}
\label{gg-rmk-resolvent-tschirnhaus}
In Step 1, a Tschirnhaus transformation is any randomly selected polynomial in \(\ZZ[x]\). Its purpose is to ensure that \(R(t)\) is squarefree. Indeed, if \(R(t)\) is not squarefree, then there is some coincidence between its roots, and therefore some unintended structure between the roots of \(F\). By transforming the roots, we should destroy this structure.

Such a transformation always exists \cite{Girstmair83}. In practice, it suffices to use \(T(x) = x\) initially, and thereafter to choose a random polynomial of small degree and coefficients, increasing the degree and coefficient bound at each iteration.
\end{remark}

\begin{remark}
\label{gg-rmk-resolvent-complex-precision}
It is important in Step 2 that we choose a complex floating point precision \(k\) such that the rounding step produces the correct answer. We do this as follows.

First, we find an upper bound on the absolute valuations of the roots of \(c(\cF)\) for each complex embedding \(c\). In principle this could be done by analyzing the polynomials which define the global model and bounding their roots in terms of the coefficients, but in our current implementation we instead compute the complex roots to some default precision (30 decimal digits) and take the size of the largest root as our bound. It is possible although unlikely that the latter approach introduces enough precision error that this bound is incorrect, and hence this part of the implementation does not yield proven results.

Using this upper bound, we can follow through the computation of \(\tilde R_c\) to get upper bounds on its coefficients. By increasing the bounds by a small fraction at each computation, we can absorb the effect of any complex precision error. We then select a precision so that the absolute errors on the coefficients \(\tilde R_{c,i}\) are less than half the shortest distance between two elements of the Minkowski lattice. We then add a generous margin to the precision (say 20 decimal digits) so that we can check in the code that we are in fact very close (say within 10 decimal digits) of an integer point.
\end{remark}

\begin{remark}
The choice to approximate the roots of \(\cF\) in the complex field \(\CC\) is somewhat arbitrary. We could instead pick a prime \(\ell\) such that \(\cF\) has a small splitting field over \(\QQ_\ell\) and approximate the roots \(\ell\)-adically. Making such a change usually improves the reliablility and precision requirements. The theory of the Minkowski lattice carries over into this setting.
\end{remark}

\section{Global model algorithms}
\label{gg-sec-glomod}

Given a polynomial \(F(x) \in K[x]\) and a global model \(i:\cK \to K\), a global model algorithm computes a global model \(\cF(x)\) for \(F(x)\) extending \(\cK\). It also computes an overgroup \(W\) such that \(G \subgrp \Gal(\cF / \cK) \subgrp W\).

\begin{remark}
As presented, these constructions assume the global model index (\cref{gg-rmk-glomodidx}) is 1, but do generalize. See \cite[Ch. II, \S4]{DPhD} for details.
\end{remark}

\subsection{\texttt{Symmetric}}

Given irreducible \(F(x) \in K[x]\), this finds a polynomial \(\cF(x) \in \cK[x]\) sufficiently close to \(F(x)\) that they have the same splitting field over \(K\). Generically we expect that \(\Gal(\cF / \cK) = S_d\), since we are not imposing any further restriction of \(\cF\), and therefore the corresponding overgroup is taken to be \(W=S_d\).

To find such a polynomial, we pick some precision parameter \(k \in \NN\). We take some polynomial \(\cF(x) \in \cK[x]\) such that \(i(\cF(x)) - F(x)\) has coefficients of valuation at least \(k\), and then we check that \(\cF\) is a global model. If not, we increase \(k\). By keeping \(k\) small, we limit the size of the coefficients of \(\cF\), which in turn limits the precision required in the complex arithmetic later.

\subsection{\texttt{Factors}}

This factorizes \(F(x) = \prod_k F_k(x)\) into irreducible factors over \(K\), produces a global model \(\cF_k(x)\) for each factor, and then the global model is \(\cF(x) = \prod_k \cF_k(x)\). The overgroup is the direct product \(W=\prod_k W_k\) of overgroups for each factor.

A parameter determines how to compute a global model for each factor.

\subsection{\texttt{RamTower}}
\label{gg-sec-ramtower}

Assuming \(F(x)\) is irreducible and defines an extension \(L/K\), this finds the ramification filtration \(L=L_t/\ldots/L_0=K\) of \(L/K\). For each segment \(L_k/L_{k-1}\), it produces a global model extending the global model of the segment below it. Then the global model is the final model in this iteration. The overgroup is the wreath product \(W=W_t\wr\cdots\wr W_1\) of overgroups of each segment.

A parameter determines how to compute a global model for each segment.

\subsection{\texttt{RootOfUnity}}
\label{gg-sec-rootofunity}

Assuming the splitting field \(L\) of \(F\) over \(K\) is unramified, and therefore generated by a primitive \(n\)th root of unity \(\zeta\), we define the global model to be \(\cL = \cK(\zeta)\).

We naturally identify \(\cW = \Gal(\cL / \cK)\) with a subgroup of \((\ZZ / n \ZZ)^\times\), identifying \(i \bmod n\) with \(\zeta \mapsto \zeta^i\). The subgroup \(W = \angles{q} \subgrp \cW\) is the decomposition group, i.e. \(\Gal(L/K)\). If \(W=\cW\) then this is our overgroup (otherwise \(\cW\) is an overgroup for a model of higher index \cite[Ch. II, \S4.5]{DPhD}).

By default, we use \(n=q^d-1\). A parameter can change this to use the smallest divisor of \(q^d-1\) not dividing \(q^c-1\) for any \(c<d\).

Another parameter controls whether to search for a \define{complement} to \(W\) --- i.e. a subgroup \(H \subgrp \cW\) such that \(H \cap W = 1\) --- of smallest index possible, and then replace \(\cL\) by the fixed field of \(H\). By design, this still has a completion to \(L\), but is of smaller degree. If \(\angles{H,W}=\cW\) then \(H\) is a \define{perfect complement} and \(W=\cW/H\) is our overgroup (otherwise \(\cW/H\) is an overgroup for a model of higher index).

\begin{remark}
The complement option usually finds a perfect complement. For example, suppose \(\cK = \QQ\) and \(K = \QQ_p\), \(p \leq 7\) and \(d \leq 50\), then there is a perfect complement unless: \(p=2\) and \(8 \mid d\); or \(p=3\) and \(d=9\); or \(p=7\) and \(d\in\braces{5,8}\).
\end{remark}

\begin{remark}
\label{gg-rmk-grunwaldwang}
The Grunwald--Wang theorem of class field theory \cite[Ch. X, \S2]{ATCFT} implies that if \(K\) is a completion \(\cK_\fp\), and \(L/K\) is cyclic, degree \(d\), then there is \(\cL / \cK\) cyclic of degree \(d\) which completes to \(L\). There is an exception at primes \(\fp \mid 2\) and degrees \(8 \mid d\), for which \((\cL : \cK) = 2d\) is sometimes necessary.
\end{remark}

\subsection{\texttt{RootOfUniformizer}}
\label{gg-sec-rootofuniformizer}

Assuming \(F\) is irreducible of degree \(d\) over \(K\) and defines a totally tamely ramified extension \(L/K\), then  \(L=K(\sqrt[d]{\pi})\) for some uniformizer \(\pi \in K\). Taking a sufficiently precise approximation to \(\pi\), we may assume that \(\pi \in \cK\), and we define the global model to be \(\cL = \cK(\sqrt[d]{\pi})\). The embedding \(\cK \to K\) extends uniquely to \(\cL \to L\).

Letting \(\zeta\) be a primitive \(d\)th root of unity, then clearly \(\cK(\sqrt[d]{\pi},\zeta)\) is the normal closure and its Galois group \(W\) (which is a function of \(\Gal(\cK(\zeta)/\cK)\) which may be computed explicitly) acts faithfully on the \(d\) elements \(\sqrt[d]{\pi}\), \(\zeta \sqrt[d]{\pi}\), \(\ldots\), \(\zeta^{d-1} \sqrt[d]{\pi}\).

\subsection{\texttt{SinglyWild}}
\label{gg-sec-singlywild}

Suppose \(F(x) \in K[x]\) defines a singly wildly ramified extension \(L/K\) of degree \(d=p^k\). That is, a totally wildly ramified extension whose ramification polygon has a single face. 

Suppose also \(p=2\) and \(L/K\) is Galois, then \(\Gal(L/K)\isom C_2^k\) and so \(L = K(\sqrt{a_1},\ldots,\sqrt{a_k})\) for some \(a_i\in K\). By taking sufficiently precise approximations, we may further assume \(a_i\in\cK\). Then \(\cL=\cK(\sqrt{a_1},\ldots,\sqrt{a_k})\) is our global model with overgroup \(W=C_2^k\).

\begin{remark}
Using Kummer theory, an averaging argument, and a result of Greve \cite[Thm. 7.3]{GP}, this method generalizes to \(p\ne2\) and non-Galois \(L/K\) \cite[Ch. II, \S4.7]{DPhD}. This has not yet been implemented.
\end{remark}

\subsection{\texttt{Select}}

This selects between several different global model algorithms, depending on \(F\). For example, we can select between \texttt{RootOfUnity}, \texttt{RootOfUniformizer} or \texttt{SinglyWild} depending on whether \(F\) defines an unramified, tame, or wild extension.

\section{Group theory algorithms}
\label{gg-sec-groups}

The job of a group theory algorithm is to decide, given the overgroup \(W\), which subgroups \(U \subgrp W\) to form resolvents from, and to use those resolvents to deduce the Galois group \(G \subgrp W\).

We recommend now reading the definition of statistic at the start of \cref{gg-sec-statistic}. A statistic is our means of comparing groups with resolvents.

\subsection{\texttt{All}}
\label{gg-sec-groups-all}

This algorithm proceeds by writing down all possible Galois groups \(G\) (up to \(W\)-conjugacy), and then eliminating possibilities until only one remains. 

There are two parameters, a statistic algorithm \(s\) (\cref{gg-sec-statistic}) which determines which properties of the Galois groups \(G\) and resolvents \(R\) to compare, and a subgroup choice algorithm (\cref{gg-sec-choice}) which determines how we choose a subgroup \(U\).

The subgroup choice algorithm is used to choose a subgroup \(U\). Then, given a resolvent \(R\), we compute the statistic \(s(R)\) and see for which \(G\) in the list of possible Galois groups this equals \(s(q(G))\) where \(q\) is the coset action of \(W\) on \(W/U\). We eliminate the \(G\) for which the statistics differ. We are done when only one \(G\) remains.

\begin{remark}
The parameters must be chosen correctly to ensure that the algorithm terminates, otherwise it is possible that the subgroup choice algorithm cannot find a useful subgroup for the given statistic. \Cref{gg-lem-rootsmax} below implies the algorithm terminates for the \texttt{HasRoot} statistic (or any more precise statistic such as \texttt{FactorDegrees}) and any subgroup choice algorithm which considers all groups.
\end{remark}

\begin{lemma}
\label{gg-lem-rootsmax}
\(G\) is congruent to a subgroup of \(U\) if and only if the corresponding resolvent \(R\) has a root.
\end{lemma}

\begin{proof}
\(G \subgrp U\) if and only if \(q(G)\) has a fixed point, where \(q:W \to S_{W/U}\) is the coset action. Since \(\Gal(R)=q(G)\), this occurs if and only if \(R\) has a root.
\end{proof}

\subsection{\texttt{Maximal}}
\label{gg-sec-groups-maximal}

This algorithm avoids the need to enumerate all possible Galois groups. We start at the top of the directed acyclic graph of subgroups of \(W\) and work our way down, at each stage either proving that a current group under consideration is not the Galois group, and so moving on to its maximal subgroups, or proving that the Galois group is not a subgroup of some of the maximal subgroups of a group under consideration.

Specifically, at all times we have a set \(\cP\) of subgroups of \(W\) such that we know that the Galois group is contained in at least one of them. We call this the \define{pool}. Initially we have \(\cP = \{W\}\). If for some resolvent \(R\) and \(P \in \cP\) we find that their statistics do not agree, i.e. \(s(R) \not\sim s(q(P))\), then we record that \(G \ne P\). We also test if the statistic is consistent with the Galois group being a subgroup of \(P\). If this latter test fails, i.e. \(s(R) \not\preceq s(q(P))\), then we remove \(P\) from the pool. We also perform the same tests on all maximal subgroups \(Q \subgrpne P \in \cP\).

Having processed a resolvent in this way, we may decide to modify \(\cP\) further. For example, as soon as there is some \(P \in \cP\) such that the Galois group is not \(P\), replace \(P\) by its maximal subgroups. Or instead, when all \(P \in \cP\) are known not to be the Galois group, replace the whole pool by the set of maximal subgroups of its elements. This behaviour is parameterised.

We have determined the Galois group when \(\cP\) contains one group, and we have deduced that the Galois group is not contained in any of its maximal subgroups.

The question remains of which subgroups \(U\subgrp W\) are \define{useful} in the sense that a resolvent formed from \(U\) will provide information. Unlike the \code{All} algorithm, it is not possible to determine for certain if a given group \(U\) will allow us to make progress or not. There is a necessary condition, but this does not guarantee progress, and there is a sufficient condition, but it is not guaranteed there there exists a group with this condition. We parameterise this choice, but in the next section give an improved method without this issue.

\subsection{\texttt{Maximal2}}
\label{gg-sec-groups-maximal2}

Note that a shortcoming of the \texttt{Maximal} algorithm is that it is not always possible to tell if a subgroup \(U \subgrp W\) will provide any information, and so its behaviour is more heuristic than principled. Another problem is that it only ever rules groups out of consideration which cannot contain the Galois group, and therefore all groups \(P\) with \(G \subgrp P \subgrp W\) will be considered in the pool \(\cP\) at some point; if there are many such groups, this can get inefficient. The \texttt{Maximal2} algorithm avoids both of these problems by positively identifying groups which do contain the Galois group.

As before, we have a pool \(\cP\) of subgroups, at least one of which contains the Galois group. Suppose there is a group \(U \subgrp W\) such that \(s(q(P)) \not\sim s(q(Q))\) for some \(P \in \cP\) and maximal \(Q \subgrpne P\) (such a group is \define{useful}) and we form the corresponding resolvent \(R\). There are two possibilities.

If \(s(R) \sim s(q(P))\) then \(s(q(Q)) \prec s(R)\), so \(s(R) \not\preceq s(q(Q))\), so \(G \not\subgrp Q\), and so we can rule \(Q\) out of consideration.

Otherwise \(s(R) \not\sim s(q(P))\) and so \(G \ne P\). In the \texttt{Maximal} algorithm at this point we would do something like replace \(P\) in the pool by its maximal subgroups. Instead, we find the set \(X''\) of subgroups \(Q'' \subgrpne q(P)\) which are maximal among those such that \(s(Q'') \sim s(R)\); we refer to these as the \define{maximal preimages in \(q(P)\) of \(s(R)\)}. Then we let \(X = \{P \cap q^{-1}(Q'') \suchthat Q'' \in X''\}\). By construction, if \(G \subgrp P\) then \(G \subgrp Q'\) for some \(Q' \in X\) and so we can replace \(P\) in the pool by \(X\). Typically \(X\) is much smaller than the number of maximal subgroups of \(P\).

Suppose now that we have eliminated all maximal subgroups of all \(P \in \cP\) from consideration. Then we know that \(G=P\) for some \(P \in \cP\). We are now in the scenario of the \texttt{All} algorithm, and so can now eliminate groups from the pool by finding \(U \subgrp W\) such that \(s(q(P_1)) \not\sim s(q(P_2))\) for some \(P_1,P_2 \in \cP\). Such a \(U\) is also said to be \define{useful}.

We have deduced the Galois group when the pool contains a single group, and we have ruled all of its maximal subgroups out of consideration.

We can use any statistic which has an equivalence relation (as required for \texttt{All}) and a partial ordering (as required for \texttt{Maximal}) and an algorithm for computing maximal preimages. For the latter, in general we have a ``naive'' algorithm, which simply works down the subgroups of \(P\) until ones with the correct statistic are found.

\begin{algorithm}[Maximal preimages: Naive]
\label{gg-alg-maxpre-naive}
Given a group \(P\), a statistic \(s\) and a value \(v\) of \(s\), returns the maximal preimages of \(v\) in \(P\).
\begin{algorithmic}[1]
\If{\(v \sim s(P)\)}
  \State\Return \(\{P\}\)
\ElsIf{\(v \prec s(P)\)}
  \State\Return \(\bigcup_{\text{maximal \(Q \subgrpne P\)}} \text{maximal preimages of \(v\) in \(Q\)}\)
\Else
  \State\Return \(\emptyset\)
\EndIf
\end{algorithmic}
\end{algorithm}

However, only using the naive algorithm would not provide an improvement over \code{Maximal}. The real efficiency gain comes from the existence of more efficient algorithms for particular statistics, in particular \texttt{HasRoot} (\cref{gg-sec-stat-hasroot}) and \texttt{FactorDegrees} (\cref{gg-sec-stat-facdegs}).

\subsection{\texttt{Sequence}}

This takes as parameters a sequence of group theory algorithms. Each one is used in turn until either the Galois group is deduced or the subgroup choice algorithm runs out of subgroups to try.

If the same algorithm appears consecutively with different parameters, then the state of the algorithm (such as the pool of possible Galois groups) is maintained so that information is not lost.

This allows us, for example, to first use a cheap statistic on a limited number of subgroups --- aiming to deduce easy Galois groups quickly --- before trying a more expensive statistic.

\section{Statistic algorithms}
\label{gg-sec-statistic}

A statistic algorithm is a means of comparing the Galois group of a polynomial with a permutation group. Specifically it is a function which takes as input a permutation group or a polynomial and outputs some value. There must be an equivalence relation on these values, which we denote \(\sim\). A statistic function \(s\) must satisfy the following property: \(s(R) \sim s(\Gal(R))\) for all polynomials \(R\). For most statistics, \(\sim\) is equality.

Using this, if we are given a polynomial \(R(x)\) (such as a resolvent) and a permutation group \(G\) and we find that \(s(R) \not\sim s(G)\), then we know that \(\Gal(R) \neq G\). This is the basis of the \texttt{All} (\cref{gg-sec-groups-all}) group theory algorithm.

Optionally, statistics can also support a partial ordering, denoted \(\preceq\), which must respect the partial ordering due to subgroups. Specifically, the following must hold: for all groups \(G,H\), if \(H \subgrp G\) then \(s(H) \preceq s(G)\). Statistics supporting this operation may be used in the \texttt{Maximal} (\cref{gg-sec-groups-maximal}) and \texttt{Maximal2} (\cref{gg-sec-groups-maximal2}) group theory algorithms.

Optionally, ordered statistics can also provide a specialised algorithm to compute maximal preimages, as defined in \cref{gg-sec-groups-maximal2}.

\subsection{\texttt{HasRoot}}
\label{gg-sec-stat-hasroot}

\(s(G)\) is true if it has a fixed point, and otherwise is false. Correspondingly, \(s(R)\) is true if it has a root (in its base field \(K\)).

If \(H \leq G\) and \(G\) has a fixed point, then so does \(H\), so we define \(v_1 \preceq v_2\) to be \(v_2 \implies v_1\).

The maximal subgroups with a fixed point are point stabilizers. Two point stabilizers are conjugate if they stabilize a point in the same orbit, and so we deduce the following algorithm to compute maximal preimages.

\begin{algorithm}(Maximal preimages: \texttt{HasRoot})
\label{gg-alg-maxpre-hasroot}
Given a group \(P\) and a value \(v \in \{\text{true},\text{false}\}\), returns the maximal preimages of \(v\) in \(P\).
\begin{algorithmic}[1]
\If{\(v = \text{true}\)}
  \State\Return \(\{\Stab_P(x) \text{ for some \(x \in o\)} \suchthat o \in \Orbits(P)\}\)
\Else
  \State\Return \(\{P\}\)
\EndIf
\end{algorithmic}
\end{algorithm}

\subsection{\texttt{NumRoots}}
\label{gg-sec-stat-numroots}

\(s(G)\) is the number of fixed points of \(G\). Correspondingly, \(s(R)\) is the number of roots of \(R\).

If \(H \leq G\) then \(H\) has at least as many fixed points as \(G\), so \(\preceq\) in this case is the usual \(\leq\) on integers.

\subsection{\texttt{Factors}}
\label{gg-sec-stat-factors}

This takes a parameter, which is another statistic \(s'\). Then \(s(G)\) is the multiset \(\{s'(G')\}\) where \(G'\) runs over the images of \(G\) acting on each of its orbits (so the degree of \(G'\) is the size of the corresponding orbit). Correspondingly, \(s(R)\) is the multiset \(\{s'(R')\}\) where \(R'\) runs over the irreducible factors of \(R\).

\subsection{\texttt{Degree}}
\label{gg-sec-stat-degree}

\(s(G)\) is the degree of the permutation group \(G\) and \(s(R)\) is the degree of \(R\).

If \(H \leq G\), then they are permutation groups of equal degree, so \(v_1 \preceq v_2\) is \(v_1 = v_2\).

\subsection{\texttt{FactorDegrees}}
\label{gg-sec-stat-facdegs}

\(s(G)\) is the multiset of sizes of orbits of \(G\). Correspondingly, \(s(R)\) is the mulitset of degrees of irreducible factors of \(R\).

This is equivalent to \code{Factors} with the \code{Degree} parameter, but is more efficient because it does not require the explicit computation of the orbit images of \(G\) on its orbits.

Additionally, it supports ordering as follows: we know that if \(H \leq G\) then the orbits of \(H\) form a refinement of the orbits of \(G\); that is, the orbits of \(G\) are unions of orbits of \(H\). Hence, given two multisets \(v_1\) and \(v_2\) of orbits sizes, we check combinatorially if one is a refinement of the other.

We provide an algorithm to compute maximal preimages of this statistic. First, in case the group \(G\) is intransitive, we embed \(G\) into a direct product \(D\) and find maximal preimages there. For each preimage \(H\), and \(d \in D\) we see if any \(H^d \cap G\) is a preimage. Observing that if \(n \in N_D(H)\) and \(g \in G\) then \(H^{ndg} \cap G = (H^d \cap G)^g\), it suffices to only consider coset representatives of \(N_D(H) \backslash D / G\).

\begin{algorithm}[Maximal preimages: \texttt{FactorDegrees}]
\label{gg-alg-maxpre-facdegs}
Given a group \(G\) of degree \(d\) and a multiset \(v\) of integers such that \(\sum v = d\), returns all maximal preimages of \(v\) in \(G\) up to conjugacy.
\begin{algorithmic}[1]
\State \(S \leftarrow \emptyset\)
\State Embed \(G \subset D = G_1 \times \ldots \times G_r\) 
\For{maximal preimages \(H\) of \(v\) in \(D\) (\cref{gg-alg-maxpre-facdegs-dp})}
  \For{double coset representatives \(d\) of \(N_D(H) \backslash D / G\)}
    \State \(H' \leftarrow H^d \cap G\)
    \If{\(H'\) has orbits of sizes \(v\)}
      \State \(S \leftarrow S \cup \{H'\}\)
    \EndIf
  \EndFor
\EndFor
\State \Return \(S\)
\end{algorithmic}
\end{algorithm}

To find maximal preimages in direct products, we first find all the ways in which \(v\) may be written as a union, with each component corresponding to a direct factor. Then by \cref{gg-lem-subpar-dp}, the maximal preimages in \(D\) are direct products of the maximal preimages in each (transitive) factor.

\begin{algorithm}[Maximal preimages: \texttt{FactorDegrees}: Direct products]
\label{gg-alg-maxpre-facdegs-dp}
Given a direct product \(G = G_1 \times \ldots \times G_r\) and \(v\) as above, returns all maximal preimages of \(v\) in \(G\) up to conjugacy.
\begin{algorithmic}[1]
\State \(S \leftarrow \emptyset\)
\For{multisets \((v_1,\ldots,v_r)\) of integers such that \(\sum v_i = \deg G_i\) and \(\bigcup_i v_i = v\)}
  \For{i = 1, \ldots, r}
    \State \(S_i \leftarrow\) maximal preimages of \(v_i\) in \(G_i\) (\cref{gg-alg-maxpre-facdegs-trans})
  \EndFor
  \For{\((H_1,\ldots,H_r) \in \prod_i S_i\)}
    \State \(S \leftarrow S \cup \{H_1 \times \ldots \times H_r\}\)
  \EndFor
\EndFor
\State \Return \(S\)
\end{algorithmic}
\end{algorithm}

To find maximal preimages in transitive groups, we embed \(G\) into a wreath product \(W\), and solve the problem there. As with \cref{gg-alg-maxpre-facdegs}, a loop over coset representatives lifts these to all preimages in \(G\).

\begin{algorithm}[Maximal preimages: \texttt{FactorDegrees}: Transitive]
\label{gg-alg-maxpre-facdegs-trans}
Given a transitive group \(G\) and \(v\) as above, returns all maximal preimages of \(v\) in \(G\) up to conjugacy.
\begin{algorithmic}[1]
\State \(S \leftarrow \emptyset\)
\State Embed \(G \subset W = G_r \wr \ldots \wr G_1\) 
\For{maximal preimages \(H\) of \(v\) in \(W\) (\cref{gg-alg-maxpre-facdegs-wr})}
  \For{double coset representatives \(w\) of \(N_W(H) \backslash W / G\)}
    \State \(H' \leftarrow H^w \cap G\)
    \If{\(H'\) has orbits of sizes \(v\)}
      \State \(S \leftarrow S \cup \{H'\}\)
    \EndIf
  \EndFor
\EndFor
\State \Return \(S\)
\end{algorithmic}
\end{algorithm}

\begin{remark}
Sometimes, if the wreath product \(W\) is very large compared to \(G\), the number of double cosets to check makes \cref{gg-alg-maxpre-facdegs-trans} infeasible. In this case, we use the naive algorithm instead.
\end{remark}

For wreath products, we work recursively so that we only need to consider a single wreath product \(A \wr B\). By \cref{gg-lem-subpar-wr}, the maximal preimages correspond to choosing a partition \(\cX\) for \(B\), and for each \(X \in \cX\) a partition \(\cY_X\) for \(A\), with \(v = \{\abs X \abs Y : Y \in \cY_X, X \in \cX\}\). We can think of \(v\) as the areas of a \(d \times e\) rectangle which has a series of vertical cuts (corresponding to the sizes of \(\cX\)), and each piece (\(X\)) having a further series of horizontal cuts (corresponding to the sizes of \(\cY_X\)). We call this a ``rectangle division'' (see \cref{gg-fig-recdiv}). For each such division, we find all possible corresponding partitions of \(A\) and \(B\), and take all combinations to construct the partitions for \(A \wr B\).

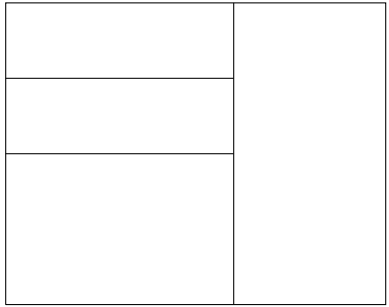
\begin{figure}
\centering
\begin{tikzpicture}
\draw (0,0) -- (5,0) -- (5,4) -- (0,4) -- (0,0);
\draw (3,0) -- (3,4);
\draw (0,2) -- (3,2);
\draw (0,3) -- (3,3);
\end{tikzpicture}
\caption[A rectangular division]{A rectangular division of a \(5 \times 4\) rectangle, represented as \(\{(3, \{2,1,1\}), (2, \{4\})\}\), with areas \(\{8,6,3,3\}\).}
\label{gg-fig-recdiv}
\end{figure}

\begin{algorithm}
\label{gg-alg-maxpre-facdegs-wr}
Given a wreath product \(G = W_r \wr \ldots \wr W_1\) and \(v\) as above, returns all maximal preimages of \(v\) in \(G\) up to conjugacy.
\begin{algorithmic}[1]
\If{r = 0}
  \State \Return \(\{G\}\)
\EndIf
\State \(A \leftarrow W_r \wr \ldots \wr W_2\)
\State \(B \leftarrow W_1\)
\State \(S \leftarrow \emptyset\)
\For{rectangle divisions \(\{(w_i,\{h_{i,j} \suchthat j\}) \suchthat i\}\) of \(\deg A \times \deg B\) into areas \(v\) 
}
  \State \(S_B \leftarrow\) maximal preimages of \(\{w_i \suchthat i\}\) in \(B\) (naive \cref{gg-alg-maxpre-naive})
  \For{i}
    \State \(S_{A,i} \leftarrow\) maximal preimages of \(\{h_{i,j} \suchthat j\}\) in \(A\) (recursively)
  \EndFor
  \For{\(H_B \in S_B\)}
    \State \(\cX \leftarrow \Orbits(H_B)\)
    \For{bijections \(m : \cX \to \{i\}\) so that \(\abs X = w_{m(X)}\)}
      \For{\((H_{A,1},\ldots) \in \prod_i S_{A,i}\)}
        \State \(H \leftarrow \parens{\prod_x H_{A,m(\cX(x))}} \sdp H_B\)
        \State \(S \leftarrow S \cup \{H\}\)
      \EndFor
    \EndFor
  \EndFor
\EndFor
\State \Return \(S\)
\end{algorithmic}
\end{algorithm}

We use the naive algorithm to find the maximal preimages of transitive and primitive groups. Since we are mainly dealing with groups close to \(p\)-groups, we expect that they have plenty of block structure and therefore the factors in any such wreath product are small enough to use the naive algorithm.

\subsection{\texttt{NumAuts}}
\label{gg-sec-stat-numauts}

\(s(G)\) is the index \((N_G(S):S)\) where \(S := \Stab_G(1)\), assuming \(G\) is transitive. \(s(R)\) is the number of automorphisms \(\abs{\Aut(L/K)}\) where \(R\) is irreducible and defines the extension \(L/K\).

Observe that if \(G=\Gal(R/K)\), then \(S=\Gal(R/L)\), \(N_G(S)\) is (by definition) the largest subgroup of \(G\) in which \(S\) is normal, and hence its fixed field is the smallest subfield \(M\) of \(L/K\) such that \(L/M\) is normal. Hence \(\Gal(L/M)\) is \(\Aut(L/K)\), and so \(\Aut(L/K) \isom N_G(S)/S\).

As we shall see in \cref{gg-lem-autgrp-order}, if \(H \leq G\) then \(s(G) \mid s(H)\). Hence \(v_1 \preceq v_2\) is \(v_2 \mid v_1\).

\subsection{\texttt{AutGroup}}
\label{gg-sec-stat-autgroup}

\(s(G)\) is the group \(N_G(S)/S\) where \(S := \Stab_G(1)\) as a regular permutation group of degree \((N_G(S):S)\); it requires \(G\) to be transitive. Correspondingly, \(s(R)\) requires \(R\) to be irreducible, and is \(\Aut(L/K)\) where \(L\) is the field defined by \(R\).

\(v_1 \sim v_2\) iff \(v_1\) and \(v_2\) are groups of the same degree and are conjugate in the symmetric group of this degree.

The test for ordering uses the following lemma, which says that as the Galois group gets smaller, the automorphism group gets larger. Hence \(v_1 \preceq v_2\) is defined as follows: \(v_1\) must have degree at least the degree of \(v_2\), and \(v_2\) must be conjugate to a subgroup of \(v_1\).

\begin{lemma}
\label{gg-lem-autgrp-order}
Suppose \(G' \leq G\) acts transitvely on a set \(X\). Fix \(x \in X\) and define \(S := \Stab_G(x)\), \(N := N_G(S)\), \(A := N/S\) and define \(S'\), \(N'\), \(A'\) similarly with respect to \(G'\). Then \(A\) is naturally isomorphic to a subgroup of \(A'\).
\end{lemma}

\begin{proof}
By definition
\begin{align*}
N &= \{ n \in G : s \in S \implies s^n \in S \} \\
  &= \{ n \in G : s \in S \implies (s^n)(x) = x \} \\
  &= \{ n \in G : s \in S \implies s(n(x)) = n(x) \} \\
  &= \{ n \in G : s \in S \implies s \in \Stab_G(n(x)) \} \\
  &= \{ n \in G : S \subseteq \Stab_G(n(x)) \} \\
  &= \{ n \in G : S = \Stab_G(n(x)) \} \text{  by orbit-stabilizer theorem} \\
  &= \{ n \in G : n(x) \in \Fix(S) \} \\
  &= \{ n \in G : n(y) \in \Fix(S) \} \text{ for any \(y \in \Fix(S)\) by symmetry} \\
  &= \{ n \in G : y \in \Fix(S) \implies n(y) \in \Fix(S) \} \\
  &= \Stab_G \Fix(S)
\end{align*}
is the group of elements of \(G\) which permute the fixed points of \(S := \Stab_G(x)\).

Since \(G\) is transitive, for each \(y \in \Fix(S)\) there exists \(n \in G\) such that \(n(x) = y\), and hence \(n \in N\). We deduce that \(N\) acts transitively on \(\Fix(S)\), and in particular the orbit-stabilizer theorem implies that \[\abs A = (N:S) = \abs{\Fix(S)}.\]

Similarly, since \(G'\) is also transitive then \(N \cap G' = \Stab_{G'} \Fix(S)\) acts transitively on \(\Fix(S)\), and so the orbit-stabilizer theorem implies \[\abs{N \cap G'} = \abs{\Stab_{N \cap G'}(1)} \abs{\Fix(S)},\] but noting that the stabilizer is actually \(S'\) then we deduce \[(N \cap G' : S') = (N : S).\]

The isomorphism theorems imply \[(N \cap G')/(S \cap G') \isom (N \cap G')S/S \leq N/S,\] but noting that \(S' = S \cap G'\) then the previous paragraph implies that we have equality, and hence naturally \[(N \cap G')/(S \cap G') \isom N/S =: A.\]

Finally, note that \[N \cap G' = \Stab_{G'} \Fix(S) \leq \Stab_{G'} \Fix(S') =: N'\] so that \[(N \cap G')/(S \cap G') \leq N'/S' =: A'.\]
\end{proof}

\subsection{\texttt{Tup}}

This statistic takes as a parameter a tuple \((s_1,\ldots,s_k)\) of statistic algorithms. Then \(s(G) = (s_1(G),\ldots,s_k(G))\) and similarly for \(s(R)\). Also \(v_1 \sim v_2\) iff \(v_{1,i} \sim v_{2,i}\) for all \(i\), and similarly for \(\preceq\).

\section{Subgroup choice algorithms}
\label{gg-sec-choice}

A subgroup choice algorithm decides, given the current state of a group theory algorithm (\cref{gg-sec-groups}) for the resolvent method, which subgroup \(U \subgrp W\) to form a resolvent from next.

Currently we use one method \code{Tranche} which generates a sequence \(\sU_1,\sU_2,\ldots\) of sets of subgroups of \(W\) one at a time, which we call \define{tranches}. Given the current tranche, \(\sU\), we inspect each element \(U\) in turn to test if it is useful by some measure (see \cref{gg-rmk-useful}). If so, we use one such \(U\). If there is no such \(U\), we declare the tranche useless and move on to the next one.

The idea is that we avoid enumerating all possible subgroups \(U\subgrp W\), and only generate them until we find a useful one.

\begin{remark}[On usefulness]
\label{gg-rmk-useful}
In the \texttt{All} group theory algorithm, we have a pool \(\cP\) of all possible Galois groups, and therefore we know all of the possible outcomes of using the group \(U\) to form a resolvent: i.e. the resolvent has one of the Galois groups \(\braces{q(P) \,:\, P \in \cP}\) and so we measure the statistic values \(\cS = \braces{s(q(G)) \,:\, P \in \cP}\). If \(\cS\) contains multiple elements, then \(U\) is useful because we will certainly cut down the list \(\cP\). Usefulness for \texttt{Maximal} and \texttt{Maximal2} is defined in \cref{gg-sec-groups-maximal,gg-sec-groups-maximal2}.
\end{remark}

The rest of this section describes some possible methods for producing tranches.

\subsection{\texttt{All}}
\label{gg-sec-tranche-all}

Produces a single tranche containing all subgroups of \(W\).

\subsection{\texttt{Index}}
\label{gg-sec-tranche-idx}

For each divisor \(n \mid \abs{W}\), produces a tranche containing all the subgroups of \(W\) of index \(n\).

There are algorithms to produce the subgroups of a group with a given index. For example, the \texttt{Subgroups} intrinsic in Magma has a \texttt{IndexEqual} parameter for this purpose.

\subsection{\texttt{OrbitIndex}}
\label{gg-sec-tranche-oidx}

\begin{definition}
\label{gg-def-oidx}
For \(U \subgrp W \subgrp S_d\), the \define{orbit index of \(U\) in \(W\)} is the index \((W : U')\) where \[U' = \Stab_W \Orbits(U) = \braces{w \in W \,:\, X \in \Orbits(U), x \in X \implies w(x) \in X}\] and is denoted \((W:U)^\orb\). The \define{remaining orbit index of \(U\) in \(W\)} is \((W:U)/(W:U)^\orb = (U':U)\). If \(\cX\) is a partition of \(\braces{1,\ldots,d}\), then it is a \define{subgroup partition for \(W\)} if there exists \(U \subgrp W\) such that \(\cX = \Orbits(U)\). The \define{index} \((W:\cX)\) of a subgroup partition \(\cX\) is \((W : \Stab_W(\cX))\).
\end{definition}

For each divisor \(n \mid \abs{W}\) and \(r \mid n\), produces a tranche containing all the subgroups of \(W\) of index \(n\) and of remaining orbit index \(r\).

We find empirically that restricting to small \(r\), such as \(\val_p(r)\le1\), typically results in an algorithm which still terminates, and does so more quickly because it generates many fewer groups.

To produce the tranche corresponding to a given \((n,r)\), we compute the subgroup partitions \(\cX\) of \(\braces{1,\ldots,d}\) such that \((W:\Stab_W(\cX)) = m := \tfrac{n}{r}\), and then compute the subgroups of \(\Stab_W(\cX)\) of index \(r\). To efficiently compute the subgroup partitions of \(W\) of a given index, we use the special form of \(W\). If \(W\) is a wreath product, direct product, or symmetric group, then we can use the algorithms in the rest of this section to reduce the problem to computing subgroup partitions of smaller groups. For these smaller groups, we compute the subgroup partitions by explicitly enumerating all the subgroups.

\begin{lemma}[Partitions of direct products]
\label{gg-lem-subpar-dp}
Suppose \(W_i \subgrp S_{d_i}\) for \(i=1,\ldots,k\) (each symmetric group acting on a disjoint set) and \(W = W_1 \times \cdots \times W_k\). If \(\cX_i\) is a partition for \(W_i\) of orbit index \(m_i\) then \(\bigcup_i \cX_i\) is a partition for \(W\) of orbit index \(\prod_i m_i\). Every partition for \(W\) is of this form.
\end{lemma}

\begin{proof}
By definition \(m_i = (W_i : \Stab_W(X_i))\). Now \[\Stab_W(\bigcup_i X_i) = \prod_i \Stab_{W_i}(X_i)\] and the result follows. Take any \(U \subgrp W\), and consider its projections \(U_i\) to \(W_i\), and let \(\cX_i = \Orbits(U_i)\), then clearly \(\cX = \bigcup_i \cX_i\).
\end{proof}

\begin{algorithm}[Partitions of direct products]
\label{gg-alg-subpar-dp}
Given \(W_i \subgrp S_{d_i}\) for \(i=1,\ldots,k\) and an integer \(m \mid \prod_i \abs{W_i}\), this returns all the partitions for \(W = W_1 \times \cdots \times W_k\) of index \(m\).
\begin{algorithmic}[1]
\If{\(k = 0\)}
  \State\Return \(\braces{\emptyset}\)
\EndIf
\State \(S \leftarrow \emptyset\)
\ForAll{\(m_1 \mid \gcd(m, \abs{W_1})\)}
  \State \(S_1 \leftarrow\) partitions of \(W_1\) of index \(m_1\)
  \State \(S_2 \leftarrow\) partitions of \(W_2 \times \cdots \times W_k\) of index \(m_2 = \tfrac{m}{m_1}\)
  \State \(S \leftarrow S \cup \braces{\cX_1 \cup \cX_2 \,:\, \cX_1 \in S_1, \cX_2 \in S_2}\)
\EndFor
\State\Return \(S\)
\end{algorithmic}
\end{algorithm}

\begin{lemma}[Partitions of wreath products]
\label{gg-lem-subpar-wr}
Suppose \(A,B\) are permutation groups, let \(\cX\) be a subgroup partition for \(B\), and for each \(X \in \cX\) let \(\cY_X\) be a subgroup partition for \(A\). Then \(\cZ = \braces{X \times Y \,:\, X \in \cX, Y \in \cY_X}\) is a subgroup partition for \(W = A \wr B\), its index is \((B:\cX) \prod_{X \in \cX} (A:\cY_X)^{\abs{X}}\), and all subgroup partitions are of this form up to conjugacy.
\end{lemma}

\begin{proof}
If \(A\) acts on \(\{1,\ldots,d\}\) and \(B\) acts on \(\{1,\ldots,e\}\), then elements of \(A \wr B\) can be defined as elements of the cartesian product \(A^e \times B\) acting on \(\{1,\ldots,e\} \times \{1,\ldots,d\}\) as \[(a_1,\ldots,a_e,b)(x,y) = (b x, a_x y).\] This implies the group operation is \[(a'_1,\ldots,a'_e,b')(a_1,\ldots,a_e,b) = (a'_{b1}a_1,\ldots,a'_{bd}a_d,b'b).\]

Suppose \(\cZ\) is defined as above, and take any \((x,y),(x',y') \in X \times Y \in \cZ\). Choose \(b \in \Stab_B(\cX)\) such that \(b(x)=x'\), which is possible since \(\Stab_B(\cX)\) acts transitively on \(X\) by definition of a subgroup partition. Choose \(a_x \in \Stab_A(\cY_X)\) such that \(a_x(y)=y'\), and choose all other \(a_{x''} \in \Stab_A(\cY_{X''})\) for \(x'' \in X''\) arbitrarily (e.g. the identity). Defining \(g=(a_1,\ldots,a_e,b)\) then \(g(x,y) = (bx,a_xy) = (x',y')\) and by construction \(g \in \Stab_W(\cZ)\). We conclude that \(\Stab_W(\cZ)\) acts transitively on each element of \(\cZ\), and so \(\cZ\) is a subgroup partition of \(W\) as claimed.

Expressing \(A \wr B\) as a semidirect product \(A^e \sdp B\), then \(\Stab_W(\cZ)\) is the subgroup \[\parens{\prod_{x \in \{1,\ldots,e\}} \Stab_A(\cY_{\cX(x)})} \sdp \Stab_B(\cX)\] where \(\cX(x)\) is the \(X \in \cX\) such that \(x \in X\). The index \((W:\cZ)\) follows.

Suppose \(G \subgrp W\). We want to show that a conjugate of \(G\) has orbits of the form \(\cZ\). Letting \(\pi : A \wr B \to B\) be the natural projection \((a_1,\ldots,a_e,b) \mapsto b\), let \(\cX = \Orbits(\pi(G))\), which is a subgroup partition of \(B\). For each \(X \in \cX\), fix a representative \(x_X \in X\), and for each \(x \in X\), fix some \(g_x = (a_{x,1},\ldots,a_{x,e},b_x) \in G\) such that \(\pi(g_x)(x_X) = x\). Define \(\hat a_x = a_{x,x_X}\) and \(\hat g = (\hat a_1,\ldots,\hat a_e,id) \in W\) then by construction \[g_x^{-1} \hat g (x,y) = (x_X, y).\] Define \(\cY_X\) such that \(\{x_X\} \times Y\) is an orbit of \(S_X := \Stab_G(\{x_X\} \times \{1,\dots,d\})\) for each \(Y \in \cY_X\). We claim that \[\Orbits(G^{\hat g}) = \cZ = \{X \times Y \suchthat Y \in \cY_X, X \in \cX\}.\] Note that if \(g^{\hat g}(x,y)=(x',y')\) then \(\pi(g^{\hat g})(x)=\pi(g)(x)=x'\) and so \(\cX(x)=\cX(x')=X\) say. For any \((x,y),(x',y')\) with \(x,x'\in X \in \cX\), then there exists \(g \in G\) such that \(g^{\hat g}(x,y) = (x',y')\) iff there is \(g\) such that \((g_{x'}^{-1} g g_x) g_x^{-1} \hat g (x, y) = g_{x'}^{-1} \hat g (x', y')\), i.e. such that \((g_{x'}^{-1} g g_x)(x_X, y) = (x_X, y').\) This occurs iff there is \(g \in S_X\) such that \(g(x_X,y)=(x_X,y')\), which occurs iff \(\cY(y)=\cY(y')=Y\) say, in which case \((x,y),(x',y')\in X \times Y\). This proves the claim.
\end{proof}

\begin{algorithm}[Partitions of wreath products]
\label{gg-alg-subpar-wr}
Given \(A \subgrp S_d, B \subgrp S_e\) and an integer \(m \mid \abs{A}^e \abs{B}\), this returns all the partitions for \(A \wr B\) of index \(m\) up to conjugacy.
\begin{algorithmic}[1]
\State \(S \leftarrow \emptyset\)
\ForAll{\(m' \mid m\)}
  \State \(S' \leftarrow\) partitions for \(B\) of index \(m'\)
  \ForAll{\(\cX \in S'\)}
    \ForAll{factorizations of \(\tfrac{m}{m'}\) of the form \(\prod_{X \in \cX} m_X^{\abs{X}}\)}
      \ForAll{\(X \in \cX\)}
        \State \(S_X \leftarrow\) partitions for \(A\) of index \(m_X\)
      \EndFor
      \ForAll{\((\cY_X)_X \in \prod_X S_X\)}
        \State include \(\braces{X \times Y \,:\, X \in \cX, Y \in \cY_X}\) in \(S\)
      \EndFor
    \EndFor
  \EndFor
\EndFor
\State \Return \(S\)
\end{algorithmic}
\end{algorithm}

\begin{remark}
The preceding algorithm may produce multiple representatives per conjugacy class. With a little more care, we can return just one as follows.

Having chosen \(\cX\), we partition it into \(B\)-conjugacy classes \(\cX_i=\{X_{i,j}\}\). Then we consider all factorizations of \(m/m'\) of the form \(\prod_{\cX_i} m_i^{\abs{X_{i,1}}}\), and then all factorizations of \(m_i\) of the form \(\prod_{X_{i,j}\in\cX_i} m_{X_{i,j}}\) with \(m_{i,1} \le m_{i,2} \le \ldots\). Hence we have a factorization of \(m/m'\) of the form \(\prod_{X \in \cX} m_{X}^{\abs{X}}\) as above. Note that this includes all factorizations of this form exactly once up to reordering conjugate blocks \(X \in \cX\).

For such a factorization, we partition \(\cX_i\) further into classes \(\cX_{i,j}=\{X_{i,j,k}\}\) such that \(m_{i,j}:=m_{X_{i,j,k}}\) is constant within a class. Similar to before, we let \(S_{i,j} = \{\cY_{i,j,\ell}\}\) be all partitions for \(A\) of index \(m_{i,j}\), and consider all \((\cY_{i,j,\ell_k})_{i,j,k} \in \prod_{i,j,k} S_{i,j}\) with \(\ell_1 \le \ell_2 \le \ldots\). Note that this includes all \((\cY_X)_X \in \prod_X S_X\) as above precisely once up to reordering conjugate blocks \(X \in \cX\).

Letting \(\cZ = \{X_{i,j} \times Y \suchthat Y \in \cY_{i,j,\ell_k}\}\) be the corresponding partition, then all such \(\cZ\) are not conjugate in \(A \wr B\), and they cover all conjugacy classes up to reordering conjugate blocks of \(\cX\). Define \(S \subgrp S_d \wr S_e\) to be the group isomorphic to \(1_d \wr \prod_i 1_{\abs{\cX_i}} \wr S_{\abs{X_{i,1}}}\) which reorders conjugate blocks of \(\cX\), where \(1_d\) denotes the trivial subgroup of \(S_d\). Then we find all \(\cZ\) up to \(A\wr B\) conjugacy by finding all \(S\)-conjugates of \(\cZ\) up to \(A\wr B\) conjugacy as follows.

Let \(H_0 = \Stab_{A \wr B}(\cZ)\), then we want all \(S\)-conjugates of \(H_0\) up to \(A\wr B\) conjugacy. Note that if \(n \in N_S(H_0)\) and \(g \in A \wr B\) then \(H_0^{nsg} \sim_{A\wr B} H_0^s\) so it suffices to consider double coset representatives \(s\) of \(N_S(H_0) \backslash S / (A \wr B) \cap S\). Compute \(H_0^s\) for all such \(s\) and dedupe by \(A \wr B\)-conjugacy.

\end{remark}

\begin{lemma}[Partitions of symmetric groups]
\label{gg-lem-subpar-sym}
Any partition \(\cX\) of \(\braces{1,\ldots,d}\) is a subgroup partition for \(S_d\) and it has orbit index \(d! / \prod_{X \in \cX} \abs{X}!\).
\end{lemma}

\begin{proof}
Indeed \(\Stab_{S_d}(\cX) = \prod_{X \in \cX} S_X\).
\end{proof}

\begin{algorithm}[Partitions of symmetric groups]
\label{gg-alg-subpar-sym}
Given integers \(d \geq 0, m \mid d!\), returns all partitions for \(S_d\) of index \(m\) up to conjugacy.
\begin{algorithmic}[1]
\If{\(d=0\)}
  \State\Return \(\braces{\emptyset}\)
\EndIf
\State \(S \leftarrow \emptyset\)
\ForAll{\(d_1=0,\ldots,d\)}
  \If{\(d!/d_1!(d-d_1)! \mid m\)}
    \State \(S_2 \leftarrow\) partitions of \(S_{d-d_1}\) of index \(m d_1! (d-d_1)! / d!\) up to conjugacy
    \State \(S \leftarrow S \cup \braces{\braces{1,\ldots,d_1} \cup \cX_2 \,:\, \cX_2 \in S_2}\)
  \EndIf
\EndFor
\State\Return \(S\)
\end{algorithmic}
\end{algorithm}

\section{Implementation and results}
\label{gg-sec-implementation}

These algorithms have been implemented \cite{galoiscode} for the Magma computer algebra system \cite{magma}. Our main \code{GaloisGroup} routine takes two arguments: a polynomial over a \(p\)-adic field, and a string describing the parameterization of the algorithm to use.

Our algorithm is by design highly modular, with each piece of the parameterization as independent as possible from the rest. This means that if one has a new algorithm for evaluating resolvents for instance, one simply needs to implement this algorithm satisfying a particular interface, and then add a line of code to the parameterization parser.

The main omission from our implementation is that the \code{SinglyWild} global model algorithm is not available in full generality, which means that for wild extensions our global model will usually use symmetric groups. Over \(\QQ_2\) with a \(2 \times \ldots \times 2\) ramification filtration this is not a problem, but for coarser filtrations, \(S_8\) is much larger than \(C_2^3\) for example, and \(S_7\) is much larger than \(C_7\), and so our global models are far from optimal. A special case of \code{SinglyWild} has been implemented and is discussed specifically in \cref{gg-sec-impl-sw}.

All experiments reported on in this section were performed on a 2.7GHz Intel Xeon. Any timings are given in core-seconds. Tables of Galois groups have been produced from all runs in this section and are available from the implementation website \cite{galoiscode}.

Unless otherwise stated, all experiments use the ``exact'' \(p\)-adic polynomial type made available by the \texttt{ExactpAdics} package \cite{exactpadics}. This uses infinite-precision arithmetic and its routines are designed to give provably correct results (modulo coding errors) and hence our algorithm also yields provably correct results except for \cref{gg-rmk-resolvent-complex-precision}.

See \cite[Ch. II, \S13]{DPhD} for a more detailed account.

\subsection{Some particular parameterizations}

Six parameterizations we will consider are named A0, B0, A1, B1, A2 and B2. These parameterizations all try three algorithms in turn: \code{Tame} (\cref{gg-sec-tame}), \code{Singly\-Ra\-mi\-fied} (\cref{gg-sec-singlyramified}) and \code{ResolventMethod} (\cref{gg-sec-arm}). The resolvent method evaluates resolvents using a global model which first factorizes the polynomial, then finds the ramification tower of the field defined by each factor, then finds a global model for each segment of the tower.

For the A parameterizations, this global model is \code{Symmetric}. For the B parameterizations, we use the \code{RootOfUnity}, \code{RootOfUniformizer} or \code{Symmetric} global model, depending on whether the segment is unramified, tame or wild.

The number part of the parameterization name controls the group theory part of the algorithm. For A0 and B0, we enumerate \code{All} possible Galois groups, then eliminate candidates based on the \code{FactorDegrees} statistic for resolvents of all subgroups. For A1 and B1, we do the same except using the \code{OrbitIndex} method to only generate resolvents for subgroups whose remaining orbit index \(r\) satisfies \(v_p(r) \le 1\). For A2 and B2, instead of enumerating all possible Galois groups, we work down the graph of possibilities using \code{Maximal2}.

We shall also consider the parameterization 00, which is the same as A0, but which uses a \code{Symmetric} global model for each factor and the \code{RootsMaximal} group theory algorithm \cite[Ch. II, \S5.4]{DPhD} which mimics Stauduhar's original absolute resolvent method \cite{Stauduhar73}.

\subsection{Up to degree 12 over \texorpdfstring{\(\QQ_2\)}{Q2}, \texorpdfstring{\(\QQ_3\)}{Q3} and \texorpdfstring{\(\QQ_5\)}{Q5}}
\label{gg-sec-d12}

The local fields database (LFDB) \cite{LFDB} tabulates data about all extensions of degree up to 12 over \(\QQ_p\) for all \(p\) including a defining polynomial, residue and ramification degrees, Galois and inertia groups, and the Galois slope content which summarizes the ramification polygon of the Galois closure.

We have run our algorithm with the eight paramaterizations \texttt{Naive}, 00 and A0 to B2 on all defining polynomials from the LFDB of degrees 2 to 12 over \(\QQ_2\), \(\QQ_3\) and \(\QQ_5\). We also ran with the parameterization A0 but using Magma's default inexact polynomial representation, which does not guarantee correctness, which we denote A0*. In all cases, the Galois group agrees with that reported in the LFDB.

The mean run times of these are given in Tables~\ref{gg-tbl-d12-q2}, \ref{gg-tbl-d12-q3} and \ref{gg-tbl-d12-q5}. In each case, the times within 10\% of the smallest are shown in bold. Counts marked with an asterisk (*) represent a random sample of all possibilities. Times marked with a numeric superscript mean that the algorithm failed to find the Galois group for this many polynomials; these are not included in the mean. A dash (---) means the corresponding algorithm was not tried. A cross (\texttimes) means the corresponding runs were prohibitively slow. Times preceded by \(\approx\) are the mean of a small number of runs, the rest being prohibitively slow. This notation is reused in subsequent tables.

\begin{table}
    \centering
        \begingroup
    \newcommand{\postbox}{\makebox[0pt][l]}
    \newcommand{\prebox}{\makebox[0pt][r]}
    \newcommand{\xx}{\postbox{*}}
    \newcommand{\z}[1]{\postbox{\textsuperscript{#1}}}
    \begin{tabular}{rrrrrrrrrrr}
    \hline
    Degree    & \# & \multicolumn{9}{l}{Run time (seconds)} \\
              &      & Naive & 00 & A0*  & A0   & B0   & A1   & B1   & A2   & B2   \\
    \hline
    2            & 7     & \bfseries 0.03 & 0.07 & 0.04 & 0.07 & 0.07 & 0.07 & 0.07 & 0.06 & 0.07 \\
    3            & 2     & \bfseries 0.08 & 0.15 & 0.10 & 0.14 & 0.15 & 0.16 & 0.16 & 0.15 & 0.15 \\ 
    4            & 59    & \bfseries 0.05 & 0.16 & 0.09 & 0.19 & 0.23 & 0.19 & 0.23 & 0.19 & 0.23 \\
    \(2+2=4\)    & 28    & ---  & ---  & ---  & \bfseries 0.20 & 0.23 & \bfseries 0.19 & 0.22 & \bfseries 0.19 & 0.23 \\
    5            & 2     & \bfseries 0.08 & 0.15 & 0.10 & 0.15 & 0.15 & 0.15 & 0.16 & 0.15 & 0.15 \\
    6            & 47    & 1.32 & 0.28 & \bfseries 0.13 & 0.24 & 0.28 & 0.25 & 0.27 & 0.26 & 0.28 \\
    \(4+2=6\)    & 413   & ---  & ---  & ---  & \bfseries 0.34 & 0.39 & \bfseries 0.34 & 0.40 & \bfseries 0.35 & 0.42 \\
    \(3+3=6\)    & 3     & ---  & ---  & ---  & \bfseries 0.12 & \bfseries 0.13 & \bfseries 0.13 & \bfseries 0.12 & \bfseries 0.12 & \bfseries 0.12 \\
    7            & 2     & \bfseries 0.09 & 0.18 & 0.12 & 0.15 & 0.15 & 0.15 & 0.16 & 0.15 & 0.15 \\
    8            & 1823  & \(\approx 100\) & \(\approx 50\) & \bfseries 0.45 & 0.59 & 0.65 & 0.59 & 0.65 & 0.58 & 0.69 \\
    \(6+2=8\)    & 329   & ---  & ---  & ---  & \bfseries 0.43 & \bfseries 0.47 & \bfseries 0.43 & \bfseries 0.47 & \bfseries 0.44 & 0.49 \\
    \(4+4=8\)    & 1770  & ---  & ---  & ---  & \bfseries 0.57 & 0.65 & \bfseries 0.58 & 0.68 & \bfseries 0.56 & 0.80 \\
    9            & 3     & 0.15 & 0.15 & \bfseries 0.08 & 0.12 & 0.12 & 0.12 & 0.13 & 0.12 & 0.12 \\
    10           & 158   & \(\approx 90\) & \texttimes & \bfseries 0.32 & 0.43 & 0.47 & 0.44 & 0.48 & 0.49 & 0.48 \\
    11           & 2     & 0.46 & 0.17 & \bfseries 0.10 & 0.15 & 0.16 & 0.15 & 0.15 & 0.18 & 0.17 \\
    12           & 5493  & \texttimes & \texttimes & ---  & 1.26 & 1.31 & \bfseries 1.18 & \bfseries 1.21 & \bfseries 1.11 & 1.23 \\
    \(8+4=12\)   & 1000\xx& ---  & --- & ---  & 10.97 & 11.19 & 10.31 & 10.25\z1 & \bfseries 1.33 & 1.63 \\
    \(6+6=12\)   & 1128  & ---  & ---  & ---  & 2.56 & 1.74 & 2.54 & 1.70 & \bfseries 0.99 & \bfseries 0.96 \\
    14\postbox u & 78    & \texttimes & \texttimes & ---  & 1.45 & 3.97 & \bfseries 0.96 & 5.89 & 4.05 & 4.62 \\
    14\postbox t & 510   & \texttimes & \texttimes & ---  & 3.05 & \bfseries 1.19 & 1.73 & \bfseries 1.19 & 1.98 & \bfseries 1.14 \\
    16\postbox a & 64\xx  & \texttimes & \texttimes & ---  & 53.65 & 54.54 & 17.47\z4 & 18.21\z4 & {\bfseries 7.25}\z4 & {\bfseries 7.59}\z4 \\
    16\postbox b & 253\xx & \texttimes & \texttimes & ---  & 304.97 & 288.25 & 42.37\z7 & 34.90\z7 & {\bfseries 25.47}\z7 & 29.40\z7 \\
    16\postbox c & 130\xx & \texttimes & \texttimes & ---  & \texttimes & \texttimes & 133.29\z{23} & 195.59\z{23} & {\bfseries 115.38}\z4 & 150.83\z{23} \\
    18 & 2046 & --- & --- & --- & \(\approx100\) & \bfseries 1.80 & \(\approx75\) & \bfseries 1.73 & \(\approx35\) & \bfseries 1.70 \\
    20 & 511318 & --- & --- &  \multicolumn{7}{r}{(used several parameterizations; see \cref{gg-sec-d20}) \bfseries 10.70} \\
    22 & 8190 & --- & --- & --- & --- & \bfseries 2.90 & --- & \bfseries 2.77 & --- & \bfseries 2.90 \\
    \hline
    \end{tabular}
    \endgroup
    \caption[Timings on polynomials up to degree 22 over \(\QQ_2\)]{Mean run times for some parameterizations on polynomials defining fields of given degrees over \(\QQ_2\).}
    \label{gg-tbl-d12-q2}
\end{table}

\begin{table}
    \centering
    \begin{tabular}{rrrrrrrrrrr}
    \hline
    Deg & \# & \multicolumn{9}{l}{Run time (seconds)} \\
    & & Naive & 00 & A0* & A0 & B0 & A1 & B1 & A2 & B2 \\
    \hline
    2  &   3 & \bfseries 0.04 & 0.11 & 0.07 & 0.10 & 0.11 & 0.12 & 0.12 & 0.11 & 0.11 \\
    3  &  10 & 0.05 & 0.07 & \bfseries 0.04 & 0.06 & 0.06 & 0.06 & 0.06 & 0.05 & 0.06 \\
    4  &   5 & 0.10 & 0.10 & \bfseries 0.05 & 0.08 & 0.08 & 0.08 & 0.12 & 0.09 & 0.09 \\
    5  &   2 & \bfseries 0.08 & 0.16 & 0.10 & 0.15 & 0.16 & 0.15 & 0.16 & 0.14 & 0.16 \\
    6  &  75 & 0.66 & 0.29 & \bfseries 0.13 & 0.31 & 0.33 & 0.34 & 0.32 & 0.30 & 0.32 \\
    7  &   2 & 0.12 & 0.17 & \bfseries 0.10 & 0.15 & 0.18 & 0.19 & 0.15 & 0.16 & 0.17 \\
    8  &   8 & 0.10 & 0.09 & \bfseries 0.06 & 0.09 & 0.08 & 0.09 & 0.08 & 0.08 & 0.08 \\
    9  & 795 & \(\approx 400\) & \(\approx 100\) &  --- & \bfseries 0.63 & \bfseries 0.64 & \bfseries 0.67 & \bfseries 0.66 & \bfseries 0.66 & 0.73 \\
    10 &   6 & 0.14 & 0.09 & \bfseries 0.08 & 0.09 & 0.09 & 0.09 & 0.10 & 0.09 & 0.10 \\
    11 &   2 & 0.15 & 0.16 & \bfseries 0.11 & 0.17 & 0.17 & 0.18 & 0.19 & 0.21 & 0.20 \\
    12 & 785 & \texttimes & \texttimes &  --- & \bfseries 1.52 & \bfseries 1.57 & 1.90 & 2.24 & 2.21 & 2.54 \\
    \hline
    \end{tabular}
    \caption[Timings on polynomials up to degree 12 over \(\QQ_3\)]{Mean run times for some parameterizations on polynomials defining fields of given degrees over \(\QQ_3\). There were 11 polynomials of degree 12 for which A0, A1 and A2 did not succeed due to a bug in Magma; these are not included in timings.}
    \label{gg-tbl-d12-q3}
\end{table}

\begin{table}
    \centering
    \begin{tabular}{rrrrrrrrrrr}
    \hline
    Deg & \# & \multicolumn{9}{l}{Run time (seconds)} \\
    & & Naive & 00 & A0* & A0 & B0 & A1 & B1 & A2 & B2 \\
    \hline
    2  & 3   & \bfseries 0.04 & 0.11 & 0.07 & 0.12 & 0.28 & 0.11 & 0.11 & 0.12 & 0.11 \\
    3  & 2   & \bfseries 0.09 & 0.14 & \bfseries 0.10 & 0.15 & 0.15 & 0.15 & 0.15 & 0.20 & 0.16 \\
    4  & 7   & \bfseries 0.03 & 0.07 & 0.04 & 0.07 & 0.07 & 0.08 & 0.07 & 0.09 & 0.08 \\
    5  & 26  & 0.12 & 0.05 & \bfseries 0.02 & 0.05 & 0.06 & 0.05 & 0.06 & 0.05 & 0.06 \\
    6  & 7   & 0.07 & 0.09 & \bfseries 0.05 & 0.08 & 0.08 & 0.08 & 0.09 & 0.08 & 0.08 \\
    7  & 2   & 0.12 & 0.17 & \bfseries 0.10 & 0.15 & 0.16 & 0.16 & 0.16 & 0.15 & 0.21 \\
    8  & 11  & 0.07 & 0.09 & \bfseries 0.05 & 0.08 & 0.07 & 0.08 & 0.08 & 0.07 & 0.09 \\
    9  & 3   & 0.12 & 0.11 & \bfseries 0.09 & 0.15 & 0.13 & 0.13 & 0.13 & 0.13 & 0.13 \\
    10 & 258 & \(\approx 100\) & \texttimes & ---  & \bfseries 2.09 & \bfseries 1.93 & 3.00 & 2.76 & 16.02 & 11.87 \\
    11 & 2   & 0.15 & 0.17 & \bfseries 0.11 & 0.18 & 0.17 & 0.17 & 0.19 & 0.18 & 0.44 \\
    12 & 17  & 0.16 & \bfseries 0.08 & \bfseries 0.09 & \bfseries 0.08 & \bfseries 0.08 & \bfseries 0.08 & \bfseries 0.08 & \bfseries 0.08 & \bfseries 0.08 \\
    \hline
    \end{tabular}
    \caption[Timings on polynomials up to degree 12 over \(\QQ_5\)]{Mean run times for some parameterizations on polynomials defining fields of given degrees over \(\QQ_5\).}
    \label{gg-tbl-d12-q5}
\end{table}

Over \(\QQ_2\), we have also run the algorithm on a selection of reducible polynomials whose irreducible factors have a given set of degrees. For example, we consider all pairs \(F_1,F_2 \in K[x]\) of quadratic polynomials defining quadratic fields over \(\QQ_2\) and run the algorithm on \(F(x) = F_1(x) F_2(x+1)\). Note that the offset \(x+1\) ensures that \(F(x)\) is squarefree in case \(F_1=F_2\). Mean run times are given in \cref{gg-tbl-d12-q2}, where for example degree ``\(2+2=4\)'' means products of quadratics. 

Observe that A0* is generally faster than A0, suggesting there is some overhead due to using exact arithmetic. However, this overhead is around a factor of two in the worst case and usually less, so not too significant.

There is little variation in timings between the six parameterizations A0 to B2. This suggests that for small degrees, there is little overhead in writing down all possible Galois groups \(G \subgrp W\), or in enumerating all subgroups of \(W\) of a given index.

Unsurprisingly, the run time increases in both the degree \(d\) and in \(v_p(d)\), the latter being the number of wild ramification breaks possible.

Not displayed in the table is that the variance in these run times is low. In particular, the maximum run time is always within a factor of 3 of the mean, and is usually less.

For small degrees, the simple parameterization 00 is comparable to the other parameterizations. However it quickly becomes infeasible as the degree increases, taking for example about 50 seconds at degree 8 over \(\QQ_2\).

The same is true for the \texttt{Naive} algorithm. Indeed, for small degrees this is often the fastest but becomes infeasibly slow above degree about 10.

\subsection{Degree 14 over \texorpdfstring{\(\QQ_2\)}{Q2}}
\label{gg-sec-d14}

There are two types of wildly ramified extensions \(L/K=\QQ_2\) of degree 14: those with \(e(L/K)=2\) and those with \(e(L/K)=14\). In the former case, \(L\) is a ramified quadratic extension of the unique unramified extension \(U/K\) of degree 7. In the latter case, \(L\) is a ramified quadratic extension of the unique (tamely) ramified extension \(T=K(\sqrt[7]2)/K\) of degree 7. We refer to these as Type 14u and Type 14t respectively.

Using the \code{AllExtensions} intrinsic in Magma we have generated all such extensions up to \(K\)-conjugacy, and have run our algorithm on all of these. The timings are given in \cref{gg-tbl-d12-q2} separately for the two types.

As a point of comparison, \cite{AwtreyD14} uses a degree 364 resolvent relative to \(W=S_{14}\) and a few other invariants to compute the same Galois groups, taking around 20 hours per polynomial whereas our algorithm takes around 2 seconds. Our results are consistent with \cite[Table 3]{AwtreyD14}.

We see that for Type 14t, using a more sophisticated global model \code{Root\-Of\-Uni\-formizer} for \(T/K\) in the B parameterizations instead of \texttt{Symmetric} in the A parameterizations makes a marked improvement to the run-time. Even when we do use \texttt{Symmetric}, we get an improvement for using more sophisticated group theory, comparing A0, A1 and A2.

In contrast, for Type 14u using a more sophisticated global model \texttt{RootOfUnity} actually made the run time worse. In this case, with parameterization B0, most of the run time is spent computing complex approximations to resolvents, despite generally using fewer resolvents and using a lower complex precision. This suggests that the implementation of \texttt{RootOfUnity} needs to be optimized.

\subsection{Degree 16 over \texorpdfstring{\(\QQ_2\)}{Q2}}
\label{gg-sec-d16}

Recall (e.g. \cite{PS} or \cite{extensions}) that to an extension of \(p\)-adic fields, we can attach a ramification polygon, which is an invariant of the extension. By attaching further residual information such as the residual polynomials of each face of the ramification polygon, we can form a finer invariant.

Using the \texttt{pAdicExtensions} package \cite{extensionscode}, which implements these invariants, we generated all possible equivalence classes of the finest such invariant, called the \define{fine ramification polygon with residues and uniformizer residue} in \cite{extensions}, for totally ramified extensions of degree 16 of \(\QQ_2\).

For each class, we selected at random one Eisenstein polynomial generating a field with this invariant, giving us a sample of 447 polynomials.

We divide these polynomials into three types. Writing \(L=L_t/\ldots/L_0=K=\QQ_2\) for the ramification filtration of the field they generate, then Type 16a polynomials have \((L_i:L_{i-1})=2\) for all \(i\) (and hence \(t=4\)), Type 16b polynomials are those remaining with \((L_i:L_{i-1})\mid4\) for all \(i\), and Type 16c are the rest (so \((L_i:L_{i-1})=8\) or \(16\) for some \(i\)). There are 64, 253 and 130 polynomials of each type respectively.

In total, there are 4,008,960 degree 16 extensions of \(\QQ_2\) inside \(\bar \QQ_2\) of Type 16a, 1,857,120 of Type 16b and 155,024 of Type 16c \cite{Sinclair}.

Per an earlier remark, we do not have \texttt{SinglyWild} global models fully implemented and so use the less efficient \texttt{Symmetric} instead. We expect run times for Types 16b and 16c to be worse than Type 16a, since the former will work relative to groups like \(W=S_4 \wr S_4\) or \(S_2 \wr S_8\) which are larger than \(W=S_2 \wr S_2 \wr S_2 \wr S_2\) of the latter. We expect that with \texttt{SinglyWild} fully implemented, the overgroup for Types 16b or 16c will be smaller not larger than for Type 16a, and that Types 16b and 16c will therefore actually become the easier classes. See \cref{gg-sec-impl-sw} for some evidence supporting this claim.

Our algorithm has been run on these polynomials with the 6 parameterizations A0 to B2. \Cref{gg-tbl-d16-q2} summarizes the results, with the polynomials grouped by type. Mean timings are also given in \cref{gg-tbl-d12-q2} for comparison. Some of these runs failed to find the Galois group, because the parameterization ran out of resolvents to try; the number of failures is given in the table. The timings only include successful runs. To give an idea of the variance in run time, we report the median and maximum time as well as the mean.

\begin{table}
    \centering
    \begin{adjustbox}{max width=\textwidth}
    \begin{tabular}{lrrrrrr}
        \hline
        & A0 & B0 & A1 & B1 & A2 & B2 \\
        \hline
        \multicolumn{7}{l}{Type 16a (64 polynomials)} \\
        Number failed    & 0      & 0      & 4     & 4     & 4     & 4     \\
        Mean run time    & 53.65  & 54.54  & 17.47 & 18.21 & 7.25  & 7.59  \\
        Median run time  & 27.87  & 28.64  & 16.69 & 17.00 & 6.06  & 6.34  \\
        Maximum run time & 311.86 & 252.39 & 31.57 & 56.59 & 22.99 & 21.76 \\
        \hline
        \multicolumn{7}{l}{Type 16b (253 polynomials)} \\
        Number failed    & 0       & 0       & 7      & 7       & 7       & 7       \\
        Mean run time    & 304.97  & 288.25  & 42.37  & 34.90   & 25.47   & 29.40   \\
        Median run time  & 18.20   & 14.77   & 12.25  & 10.38   & 8.02    & 7.65    \\
        Maximum run time & 4016.19 & 3721.84 & 432.85 & 1182.44 & 1063.16 & 1616.56 \\
        \hline
        \multicolumn{7}{l}{Type 16c (130 polynomials)} \\
        Number failed    & --- & --- & 23      & 23      & 4        & 23      \\
        Mean run time    & --- & --- & 133.29  & 195.59  & 115.38   & 150.83  \\
        Median run time  & --- & --- & 10.50   & 1.58    & 1.43     & 1.36    \\
        Maximum run time & --- & --- & 2502.06 & 7949.19 & 12432.12 & 4368.25 \\
        \hline
    \end{tabular}
    \end{adjustbox}
    \caption[Timings on polynomials of degree 16 over \(\QQ_2\)]{Run times in seconds for a selection of parameterizations on a sample of polynomials defining fields of degree 16 over \(\QQ_2\) divided into three types.}
    \label{gg-tbl-d16-q2}
\end{table}

The run times are significantly higher at degree 16 than lower degrees, and there are now pronounced differences between the parameterizations, with A0 and B0 being the slowest and numbered A2 and B2 being the fastest.

As predicted, Type 16a polynomials are the fastest. For this type, the median is usually close to the mean and the maximum is not much larger, indicating this is a low-variance regime. Elsewhere, the median is smaller and the maximum is a lot higher, so the variance is greater.

\subsection{Degree 18 over \texorpdfstring{\(\QQ_2\)}{Q2}}
\label{gg-sec-d18}

Using the \texttt{pAdicExtensions} package \cite{extensionscode}, we have generated all ramification polygons of totally ramified extensions \(L/\QQ_2\) of degree 18. These have vertices of the form
\[(1,J), (2,0), (18,0)\]
where the discriminant valuation is \(18+J-1\). Note that these extensions are of the form \(L/T/\QQ_2\) where \(T/\QQ_2\) is the unique tame extension of degree 9 and \(L/T\) is quadratic.

For each polygon, we have generated a set of polynomials generating all extensions with this ramification polygon, and run our algorithm on them all with parameterizations A0 to B2. There are 2046 polynomials in total.

Mean timings are given in \cref{gg-tbl-d12-q2}. Note that the B parameterizations are far quicker than A as a result of using the \code{RootOfUniformizer} global model instead of \code{Symmetric} for \(T/\QQ_2\).

In \cref{gg-tbl-d18-q2} we give the number of polynomials for each ramification polygon (parameterized by \(J\)) and the count of the T-numbers of their Galois groups.

\begin{table}
    \centering
    \begin{tabular}{rrp{17em}}
    \hline
    \(J\) & \# & Groups \\
    \hline
    1  & 2 & \(433\), \(434\) \\
    3  & 4 & \(98\), \(101\), \(588\), \(592\) \\
    5  & 8 & \(433^2\), \(434^2\), \(588^2\), \(592^2\) \\
    7  & 16 & \(433^4\), \(434^4\), \(588^4\), \(592^4\) \\
    9  & 32 & \(45^2\), \(147^2\), \(512^{14}\), \(656^{14}\) \\
    11 & 64 & \(433^8\), \(434^8\), \(512^{16}\), \(588^8\), \(592^8\), \(656^{16}\) \\
    13 & 128 & \(433^{16}\), \(434^{16}\), \(512^{32}\), \(588^{16}\), \(592^{16}\), \(656^{32}\) \\
    15 & 256 & \(98^2\), \(101^2\), \(147^4\), \(588^{62}\), \(592^{62}\), \(656^{124}\) \\
    17 & 512 & \(433^{32}\), \(434^{32}\), \(512^{64}\), \(588^{96}\), \(592^{96}\), \(656^{192}\) \\
    18 & 1024 & \(45^{4}\), \(147^{12}\), \(512^{252}\), \(656^{756}\) \\
    \hline
    Total & 2046 & \(45^{6}\), \(98^{3}\), \(101^{3}\), \(147^{18}\), \(433^{63}\), \(434^{63}\), \(512^{378}\), \(588^{189}\), \(592^{189}\), \(656^{1134}\) \\
    \hline
    \end{tabular}
    \caption[Totally ramified Galois groups of degree 18 over \(\QQ_2\)]{Totally ramified Galois groups of degree 18 over \(\QQ_2\).}
    \label{gg-tbl-d18-q2}
\end{table}

Noting that \(L/T\) is Galois and \(T/\QQ_2\) has only the trivial automorphism, then \(\Aut(L/\QQ_2) \isom C_2\) and so each \(L/\QQ_2\) has 9 conjugates inside \(\bar\QQ_2\). The number of polynomials generated times 9 is equal to the number of extensions of degree 18 in \(\bar\QQ_2\), from which we deduce we have exactly one polynomial per isomorphism class.

\subsection{Degree 20 over \texorpdfstring{\(\QQ_2\)}{Q2}}
\label{gg-sec-d20}

As in \cref{gg-sec-d18}, we have generated all ramification polygons of totally ramified extensions \(L/\QQ_2\) of degree 20. For each we have produced a set of generating polynomials, 511,318 in total.

We have computed the Galois groups of all of these polynomials \(F(x)\), which required several parameterizations of our algorithm to cover all cases. This also occasionally required computing \(\Gal(F/K)\) where \(K = \QQ_2(\sqrt[3]2,\zeta_3)\), for which we can compute a more efficient global model than \(F/\QQ_2\), at the expense of some more group theory computation.

By \cite[Theorem 1]{Monge11} there are 259,968 isomorphism classes of such extensions \(L/\QQ_2\) so we have over-counted by a factor of about 2.

Average timings are given in \cref{gg-tbl-d12-q2} and counts of Galois groups are given in \cref{gg-tbl-d20-q2}. 

\begin{table}
    \centering
    \begin{tabular}{rp{24em}}
    \hline
    \# & Groups \\
    \hline
    511,318 & \(16^{4}\), \(18^{8}\), \(19^{6}\), \(20^{8}\), \(42^{48}\), \(61^{3}\), \(68\), \(77^{15}\), \(80^{15}\), \(129^{78}\), \(131^{90}\), \(132^{78}\), \(137^{90}\), \(173^{30}\), \(186^{120}\), \(189^{120}\), \(194^{180}\), \(195^{114}\), \(196^{180}\), \(261^{720}\), \(282^{90}\), \(305^{120}\), \(306^{105}\), \(309^{240}\), \(312^{240}\), \(317^{140}\), \(330^{35}\), \(332^{240}\), \(338^{140}\), \(351^{120}\), \(406^{630}\), \(411^{1440}\), \(416^{1440}\), \(417^{1440}\), \(419^{1440}\), \(420^{1440}\), \(422^{630}\), \(434^{720}\), \(435^{1440}\), \(437^{1440}\), \(441^{1440}\), \(443^{720}\), \(444^{562}\), \(447^{720}\), \(448^{720}\), \(449^{562}\), \(471^{85}\), \(472^{225}\), \(510^{1920}\), \(511^{2880}\), \(512^{1920}\), \(514^{1920}\), \(515^{2880}\), \(516^{2880}\), \(517^{1920}\), \(518^{2880}\), \(519^{1920}\), \(520^{1920}\), \(523^{1920}\), \(524^{1920}\), \(526^{1396}\), \(528^{840}\), \(529^{1920}\), \(530^{1920}\), \(632^{11520}\), \(633^{11520}\), \(634^{11520}\), \(678^{255}\), \(683^{675}\), \(847^{5760}\), \(850^{5760}\), \(851^{5760}\), \(854^{5760}\), \(906^{42240}\), \(907^{42240}\), \(908^{34560}\), \(909^{42240}\), \(910^{42240}\), \(911^{34560}\), \(946^{161280}\) \\
    \hline
    \end{tabular}
    \caption[Totally ramified Galois groups of degree 20 over \(\QQ_2\)]{Totally ramified Galois groups of degree 20 over \(\QQ_2\).}
    \label{gg-tbl-d20-q2}
\end{table}

\subsection{Degree 22 over \texorpdfstring{\(\QQ_2\)}{Q2}}
\label{gg-sec-d22}

As in \cref{gg-sec-d18}, we have generated all ramification polygons of totally ramified extensions \(L/\QQ_2\) of degree 22, these have vertices of the form
\[(1,J),(2,0),(22,0),\]
and for each we have produced a set of generating polynomials. Again, we have precisely one polynomial per isomorphism class, 8190 in total.

Timings with parameterizations B0 to B2 are given in \cref{gg-tbl-d12-q2} and counts of Galois groups are given in \cref{gg-tbl-d22-q2}.

\begin{table}
    \centering
    \begin{tabular}{rrl}
    \hline
    \(J\) & \# & Groups \\
    \hline
    1 & 2 & \(34\), \(35\) \\
    3 & 4 & \(34^{2}\), \(35^{2}\) \\
    5 & 8 & \(34^{4}\), \(35^{4}\) \\
    7 & 16 & \(34^{8}\), \(35^{8}\) \\
    9 & 32 & \(34^{16}\), \(35^{16}\) \\
    11 & 64 & \(6^{2}\), \(37^{62}\) \\
    13 & 128 & \(34^{32}\), \(35^{32}\), \(37^{64}\) \\
    15 & 256 & \(34^{64}\), \(35^{64}\), \(37^{128}\) \\
    17 & 512 & \(34^{128}\), \(35^{128}\), \(37^{256}\) \\
    19 & 1024 & \(34^{256}\), \(35^{256}\), \(37^{512}\) \\
    21 & 2048 & \(34^{512}\), \(35^{512}\), \(37^{1024}\) \\
    22 & 4096 & \(6^{4}\), \(37^{4092}\) \\
    \hline
    Total & 8190 & \(6^{6}\), \(34^{1023}\), \(35^{1023}\), \(37^{6138}\) \\
    \hline
    \end{tabular}
    \caption[Totally ramified Galois groups of degree 22 over \(\QQ_2\)]{Totally ramified Galois groups of degree 22 over \(\QQ_2\).}
    \label{gg-tbl-d22-q2}
\end{table}

\subsection{Degree 32 over \texorpdfstring{\(\QQ_2\)}{Q2}}
\label{gg-sec-d32}

Our algorithm can compute some non-trivial Galois groups of order 32. For example, consider \(F(x) = x^{16} + 32 x + 2\) which is Eisenstein with Galois group 16T1638 of index \(8=2^3\) in \(C_2^{\wr 4}\). Using A2, we find the Galois group of \(F(x^2)\) is 32T2583443 of index \(2^{10}\) in \(C_2^{\wr 5}\). This took about 125 seconds, which breaks down as follows.

\begin{center}
\begin{tabular}{lrr}
\hline
 & Run time (seconds) & Share of run time \\
\hline
Start resolvent algorithm & 23.28 & 18.6\% \\
Choose subgroup & 91.44 & 73.0\% \\
Compute resolvent & 1.39 & 1.1\% \\
Process resolvent & 6.84 & 5.5\% \\
Other & 2.37 & 1.9\% \\
\hline
Total & 125.32 & \\
\hline
\end{tabular}
\end{center}

Here, ``start resolvent algorithm'' includes initially factorizing the polynomial, finding the extensions defined by the factors, finding their ramification filtrations, and computing a corresponding global model. ``Choose subgroup'' means time spent by the subgroup choice algorithm choosing a subgroup \(U \subgrp W\) from which to form a resolvent. ``Compute resolvent'' is the time spent computing a resolvent \(R(x)\) given an invariant for the subgroup \(U\). ``Process resolvent'' is the time spent by the group theory algorithm deducing information about the Galois group from a resolvent, and so in particular includes finding the degrees of the factors of the resolvent and computing maximal preimages. ``Other'' is everything else, including intializing the group theory algorithm and computing invariants.

This used 104 resolvents in total: 82 of degree 2, 9 of degree 4, 7 of degree 8, 2 of degree 16 and 4 of degree 32. The maximum complex precision used was 4056 decimal digits.

The run time is dominated by time spent choosing subgroups \(U \subgrp W\), suggesting that this should be the focus for future improvement. The next most dominant part is time spent starting the resolvent algorithm, but this part is essentially independent of the Galois group. Very little time is spent actually computing resolvents, which is perhaps surprising given that this is the part spent using complex embeddings of global models.

\subsection{A special case of \texttt{SinglyWild}}
\label{gg-sec-impl-sw}

We have implemented \texttt{SinglyWild} in the special case \(p=2\) for totally wildly ramified extensions \(L/K\) which are Galois. Hence \(\Gal(L/K) \isom C_2^k\) where \((L:K)=p^k\).

We now define three more parameterizations C0, C1 and C2 which are the same as B0, B1 and B2 except that the \texttt{Symmetric} global model on the wild part is replaced by \texttt{SinglyWild}.

It is well-known (e.g. \cite[Ch. IV, \S2, Prop. 7]{SerLF}) that for such an extension \(L/K\) there is an injective group homomorphism \(\Gal(L/K) \to \FF_K^+\), and hence \(\Gal(L/K)\) is isomorphic to a subspace of \(\FF_K/\FF_p\). In particular, \((\FF_K : \FF_p) \ge k\) and so \(K/\QQ_p\) has residue degree at least \(k\).

Using the \texttt{pAdicExtensions} package \cite{extensionscode}, we have generated defining polynomials which between them generate all extensions of the form \(L/U/\QQ_2\) where \(U/\QQ_2\) is unramified of some degree and \(L/U\) is singly wildly ramified and Galois of some degree.

For example when \(k=2\) and \((U:\QQ_2)=4\), then the global model in C0 gives the overgroup \(W = C_2^2 \wr C_4\) of order \(2^{10}\), which is somewhat smaller than the overgroup \(W = S_4 \wr C_4\) of order \(2^{14} \cdot 3^4\) from B0.

We have run our algorithm with the 9 parameterizations A0 to C2 on these polynomials. Mean timings are given in \cref{gg-tbl-sw-q2}.

\begin{table}
    \newcommand{\postbox}{\makebox[0pt][l]}
    \newcommand{\prebox}{\makebox[0pt][r]}
    \newcommand{\z}[1]{\postbox{\textsuperscript{#1}}}
    \centering
    \begin{adjustbox}{max width=\textwidth}
    \begin{tabular}{rrrrrrrrrrrr}
    \hline
    Deg & \(k\) & \# & \multicolumn{9}{l}{Run time (seconds)} \\
    & & & A0 & B0 & C0 & A1 & B1 & C1 & A2 & B2 & C2 \\
    \hline
    8  & 2 & 4   & \bfseries 0.53 & \bfseries 0.56 & 0.61 & \bfseries 0.57 & \bfseries 0.57 & 0.66 & 0.62 & \bfseries 0.58 & 0.68 \\
    12 & 2 & 28  & 2.36 & 2.34 & 0.71 & 3.41 & 3.73 & \bfseries 0.64 & 4.57 & 4.79 & \bfseries 0.66 \\
    16 & 2 & 140 & \(\times\) & \(\times\) & 1.23 & 80.02\z1 & 23.27\z1 & \bfseries 0.95 & \(\times\) & \(\times\) & \bfseries 0.98 \\
    24 & 3 & 8     & \(\times\) & \(\times\) & \bfseries 12.55 & \(\times\) & \(\times\) & \bfseries 12.75 & \(\times\) & \(\times\) & \bfseries 12.42 \\
    32 & 3 & 120 & --- & --- & 40.49 & --- & --- & 31.34 & --- & --- & \bfseries 23.68 \\
    \hline
    \end{tabular}
    \end{adjustbox}
    \caption[Timings on polynomials using \texttt{SinglyWild} over \(\QQ_2\)]{Mean run times for a selection of parameterizations on polynomials defining fields of the form \(L/U/\QQ_2\) where \(U/\QQ_2\) is unramified and \(L/U\) is singly wildly ramified with Galois group \(C_2^k\). At degree 32, there were four polynomials which did not succeed due to a bug in Magma; these are not included in the mean.}
    \label{gg-tbl-sw-q2}
\end{table}

Except at degree 8, the C parameterizations are by far the quickest.

\section*{Acknowledgements}
This work was partially supported by a grant from GCHQ.

\mybibliography

\end{document}